\documentclass{amsart}

\usepackage{mathrsfs}
\usepackage[dvips]{graphicx}
\usepackage{amsmath}
\usepackage[latin1]{inputenc}
\usepackage{amsfonts}
\usepackage{amssymb}
\usepackage{graphicx,epsfig}
\usepackage{amsthm}

\usepackage{amscd}
\usepackage{dsfont}
\usepackage[all,dvips]{xy}
\usepackage{fancyhdr}
\usepackage[T1]{fontenc}

\input cyracc.def
\font\tencyr=wncyr10
\def\cyr{\cyracc\tencyr}

\newtheorem{theorem}{Theorem}[section]
\newtheorem{proposition}[theorem]{Proposition}
\newtheorem{lemma}[theorem]{Lemma}
\newtheorem{corollary}[theorem]{Corollary}

\theoremstyle{definition}
\newtheorem{definition}[theorem]{Definition}
\newtheorem{example}[theorem]{Example}
\newtheorem{remark}[theorem]{Remark}

\allowdisplaybreaks

\numberwithin{equation}{section}

\begin{document}

\title{Braided cofree Hopf algebras and quantum multi-brace algebras}

\author{Run-Qiang Jian}
\address{School of Computer Science, Dongguan University
of Technology, 1, Daxue Road, Songshan Lake, 523808, Dongguan, P.
R. China}
%\curraddr{}
\email{jian@math.jussieu.fr}

\author{Marc Rosso}
\address{D\'{e}partement de Math\'{e}matiques, Universit\'{e} Paris
Diderot (Paris 7), 175, rue du Chevaleret, 75013, Paris, France}
\email{rosso@math.jussieu.fr}

%    General info

\date{}



\begin{abstract}
We give a systematic construction of Hopf algebra structures on
braided cofree coalgebras. The relevant underlying structures are braided algebras and braided coalgebras. We provide some
interesting examples of these algebras and coalgebras related to
quantum groups. We introduce quantum multi-brace algebras which
are generalizations of both braided algebras and
$\textbf{B}_\infty$-algebras, as the natural framework. This new
subject enables one to quantize some important algebra structures
in a uniform way. Particular interesting examples are quantum
quasi-shuffle algebras.
\end{abstract}

\maketitle

\tableofcontents

\section{Introduction}
In \cite{LoR}, Loday and Ronco proved a classification theorem for
connected cofree bialgebras with analogues of the
Poincar\'{e}-Birkhoff-Witt theorem and of the Cartier-Milnor-Moore
theorem for non-cocommutative Hopf algebras. The main tool used is
the notion of $\textbf{B}_\infty$-algebra. This enables one to
investigate all associative algebra structures on $T(V)$
compatible with the deconcatenation coproduct. By using the
universal property of $T(V)$ with respect to the connected
coalgebra structure, the product can be rebuilt from the data of
some linear maps $M_{pq}: V^{\otimes p}\otimes V^{\otimes
q}\rightarrow V$ for $p,q\geq 0$. Conversely, one can construct an
associative algebra structure for such given maps under some
associativity conditions. Furthermore, with this algebra structure
and the deconcatenation coproduct, $T(V)$ becomes a bialgebra.

On the other hand, after the works on quantum groups which were
introduced by Drinfel'd \cite{D} and Jimbo \cite{J},
mathematicians began to be interested in subjects related to
braided categories. Besides the natural interest for mathematics
(see, e.g., \cite{Ka}, \cite{Tu} and the references therein), this
also brings many significant applications in mathematical physics,
for instance, in quantum field theory (see, e.g., \cite{BK} and
the references therein). For this purpose and the importance of
cofree Hopf algebras, we would like to study the braided version
of cofree Hopf algebra structures on $T(V)$. In order to do this,
we need to extend the notion of $\textbf{B}_\infty$-algebra to the
braided framework, where we use the braided coproduct instead of
the tensor deconcatenation coproduct of $T(V)\otimes T(V)$. In
contrast to the classical case, the structure map coming from
$\textbf{B}_\infty$-algebras with braided coproduct is not
associative in general. To overcome the problem, it requires some
compatibility conditions between the maps $M_{pq}$ and the
braiding. It leads to the definition of quantum multi-brace
algebras. With the product from quantum multi-brace algebra
structure, $T(V)$ becomes a "twisted" Hopf algebra in the sense of
\cite{Ro}. Quantum multi-brace algebras provide a systematic
construction of Hopf algebra structures on cofree braided
coalgebras.

Another motivation comes from works on multiple zeta values. They
led naturally to so-called quasi-shuffle algebras. Mainly, the
underlying vector space used to construct the shuffle algebra has
also an algebra structure. These algebras were first discovered by
Newman and Radford in \cite{NR}, and later studied by many
mathematicians in different aspects (see, e.g., \cite{EG},
\cite{Hof1}, \cite{Hof2}, \cite{IKZ}, \cite{Lod}, and the
references therein). For the reason mentioned in the preceding
paragraph, there were some attempts to quantize the quasi-shuffle
algebra, for instance, \cite{B} and \cite{Hof1}. We want to deform
quasi-shuffle algebras in the spirit of quantum shuffle algebras,
where the usual flip is replaced by a braiding. This way seems
more natural. But we have to impose compatibility between the
braiding and the algebra structure on the underlying vector space.
The quantum multi-brace algebras provide a good framework. At this
level, we obtain a natural framework for quantum quasi-shuffle
algebras, where the quantum multi-brace algebra structure has only
the $M_{11}$ term. It is valuable to mention that Hoffman's
q-deformation of quasi-shuffle product (\cite{Hof2}) is a special
case of quantum quasi-shuffle algebras.

Therefore, quantum multi-brace algebras allow one to quantize many
important algebra structures, such as shuffle algebras and
quasi-shuffle algebras, in a uniform way. The new object is not
just the generalization of $\textbf{B}_\infty$-algebras, but also
of braided algebras. As we know, braided algebras were introduced
in an explicit form by Baez in \cite{Ba1}, and Hashimoto and
Hayashi in \cite{HH} independently, where they were called
r-algebras and Yang-Baxter algebras respectively. These algebras
play an important role in braided categories. For instance, they
were used to construct braided Hochschild homologies (\cite{Ba2})
and they are the relevant structure between the braiding and the
multiplication in our construction of quantum quasi-shuffle
algebras. They also proved to be of interest in their own right
(see, e.g., \cite{AS}, \cite{AMS} and \cite{Ta}). But up to now,
there were few examples of these. Here we use quantum multi-brace
algebras to provide some. In particular, we show that the "upper
triangular part" of quantum groups are braided algebras.

This paper is organized as follows. In Section 2, we recall the
definitions of braided algebras and braided coalgebras. We also
study some of their properties. After recalling the construction
of braided algebras from Yetter-Drinfel'd modules with extra
natural conditions, we show that module-algebras (resp.
module-coalgebras) over a quasi-triangular Hopf algebra are
braided algebras (resp. coalgebras). Section 3 contains
interesting examples of braided algebras from quantum groups,
which are the so-called quantum shuffle algebras (introduced in
\cite{Ro}). We prove that the cotensor algebra $T_H^c(M)$ over a
Hopf algebra $H$ and an $H$-Hopf bimodule $M$ is both a braided
algebra and a braided coalgebra. As a consequence, the "upper
triangular part" $U_q^+$ of the quantized enveloping algebra with
a symmetrizable Cartan matrix is a braided algebra. In Section 4,
we define quantum multi-brace algebras and prove that their tensor
spaces have braided algebra structures. Quantum shuffle algebras
and quantum quasi-shuffle algebras are special quantum multi-brace
algebras. Finally, in Section 5, we introduce the notion of
2-braided algebras and use them to
construct quantum multi-brace algebras.\\

\noindent \textbf{Notation.} In this paper, we denote by $K$ a
ground field of characteristic 0. All the objects we discuss are
defined over $K$.

Let $(H,\bigtriangleup,\varepsilon, S)$ be a Hopf algebra. As
usual, we denote $\bigtriangleup^{(1)}=\bigtriangleup$ and
$\bigtriangleup^{(n)}=(\bigtriangleup^{(n-1)}\otimes
\mathrm{id}_H)\bigtriangleup$ for $n\geq 2$. We adopt Sweedler's
notation for coalgebras and comodules: for any $h\in H$,
$$\bigtriangleup(h)=\sum_{(h)}h_{(1)}\otimes h_{(2)},$$
and for a left $H$-comodule $(M,\rho)$ and any $ m\in M$,
$$\rho(m)=\sum_{(m)}m_{(-1)}\otimes m_{(0)},$$ where the part
$m_{(-1)}$ lies in $H$ and the part $m_{(0)}$ lies in $M$.

The symmetric group of n letters $\{1,2,\ldots,n\}$ is written by
$\mathfrak{S}_{n}$.

A braiding $\sigma$ on a vector space $V$ is an invertible linear
map in $\mathrm{End}(V\otimes V)$ satisfying the braid relation on
$V^{\otimes 3}$:
$$(\sigma \otimes
\mathrm{id}_V)(\mathrm{id}_V\otimes\sigma)(\sigma\otimes
\mathrm{id}_V)=(\mathrm{id}_V\otimes\sigma)(\sigma\otimes
\mathrm{id}_V)(\mathrm{id}_V\otimes\sigma).$$ A braided vector
space $(V,\sigma)$ is a vector space $V$ equipped with a braiding
$\sigma$. For any $n\in \mathbb{N}$ and $1\leq i\leq n-1$, we
denote by $\sigma_i$ the operator $\mathrm{id}_V^{\otimes
i-1}\otimes \sigma\otimes \mathrm{id}_V^{\otimes n-i-1}\in
\mathrm{End}(V^{\otimes n})$. For any $w\in \mathfrak{S}_{n}$, we
denote by $T_w^\sigma$ the corresponding lift of $w$ in the braid
group $B_n$, defined as follows: if $w=s_{i_1}\cdots s_{i_l}$ is
any reduced expression of $w$, where $s_{i}=(i,i+1)$, then
$T_w^\sigma=\sigma_{i_1}\cdots \sigma_{i_l}$. Sometimes we use
$T_w$ instead of $T_w^\sigma$ if there is no ambiguity.

For a vector space $V$, we denote by $\otimes$ the tensor product
within $T(V)$, and by $\underline{\otimes}$ the one between $T(V)$
and $T(V)$.

\section{Braided algebras and braided coalgebras}
We start by recalling the definitions of braided algebras and
braided coalgebras. In the following, algebras are always assumed
to be associative and unital, and coalgebras are always assumed to
be coassociative and counital.

\begin{definition}[\cite{Ba1}, \cite{HH}]1. Let $A=(A,m,\eta)$ be an algebra with product $m$ and unit $\eta$. Let $\sigma$ be a braiding on $A$. We call $(A,m,\sigma)$ a \emph{braided algebra} if the following diagram is commutative:
\[\begin{CD}
A^{\otimes 3} @>\sigma_1\sigma_2 >>A^{\otimes 3}@>\sigma_2\sigma_1 >> A^{\otimes 3}\\
@VV m\otimes \mathrm{id}_A V @VV\mathrm{id}_A\otimes m V @VV m\otimes \mathrm{id}_A V \\
A^{\otimes 2} @>\sigma >>A^{\otimes
2}@>\sigma >>A^{\otimes 2}\\
@AA\eta\otimes \mathrm{id}_{A} A @AA\mathrm{id}_{A}\otimes \eta A @AA \eta\otimes \mathrm{id}_{A} A \\
K\otimes A @>\simeq>>A\otimes K@>\simeq>> K\otimes A.
\end{CD}\]

2. Let $C=(C,\bigtriangleup ,\varepsilon)$ be a coalgebra with
coproduct $\bigtriangleup$ and counit $\varepsilon$. Let $\sigma$
be a braiding on $C$. We call $(C,\bigtriangleup,\sigma)$ a
\emph{braided coalgebra} if the following diagram is commutative:
\[\begin{CD}
C^{\otimes 3} @>\sigma_1\sigma_2>>C^{\otimes 3}@>\sigma_2\sigma_1 >> C^{\otimes 3}\\
@AA \bigtriangleup\otimes \mathrm{id}_C A @AA\mathrm{id}_C\otimes \bigtriangleup A @AA \bigtriangleup\otimes \mathrm{id}_C A \\
C^{\otimes 2} @>\sigma >>C^{\otimes
2}@>\sigma >>C^{\otimes 2}\\
@VV\varepsilon\otimes \mathrm{id}_C V @VV\mathrm{id}_C\otimes
\varepsilon V @VV \varepsilon \otimes \mathrm{id}_C V
\\K\otimes C@>\simeq>>C\otimes
K@>\simeq>> K\otimes C.
\end{CD}\]\end{definition}

These definitions give an appropriate way to extend the usual
algebra (resp. coalgebra) structure on the tensor products of
algebras (resp. coalgebras) in braided categories.

\begin{proposition}[\cite{HH}, Proposition 4.2]1. For a braided algebra $(A,m,\sigma)$ and any $i\in \mathbb{N}$, the braided vector space $(A^{\otimes i}, T^\sigma_{\chi_{ii}})$ becomes a braided algebra with product $m_{\sigma,i}=m^{\otimes i}\circ T^\sigma_{w_i}$ and unit $\eta^{\otimes i}:K\simeq K^{\otimes i}\rightarrow A^{\otimes i}$, where $\chi_{ii},w_i\in \mathfrak{S}_{2i}$ are given by \[\chi_{ii}=\left(\begin{array}{cccccccc}
1&2&\cdots&i&i+1&i+2&\cdots & 2i\\
i+1&i+2&\cdots&2i&1& 2 &\cdots & i
\end{array}\right),\] and \[w_{i}=\left(\begin{array}{ccccccccc}
1&2&3&\cdots&i&i+1&i+2&\cdots & 2i\\
1&3&5&\cdots&2i-1&2& 4 &\cdots & 2i
\end{array}\right).\]

2. For a braided coalgebra $(C,\bigtriangleup,\sigma)$, the
braided vector space $(C^{\otimes i}, T^\sigma_{\chi_{ii}})$
becomes a braided coalgebra with coproduct
$\bigtriangleup_{\sigma,i}=
T^\sigma_{w_i^{-1}}\circ\bigtriangleup^{\otimes i}$ and counit
$\varepsilon^{\otimes i}:C^{\otimes i}\rightarrow K^{\otimes
i}\simeq K$.\end{proposition}

\begin{remark}1. Any algebra (resp. coalgebra) is a braided algebra (resp. coalgebra) with the usual flip.

2. If $(A,m, \sigma)$ is a braided algebra, then so is
$(A,m,\sigma^{-1})$. Similarly, if $(C,\bigtriangleup,\sigma)$ is
a braided coalgebra, then so is $(C,\bigtriangleup,\sigma^{-1})$.

3. Let $<,>: V\times W\rightarrow K$ and
$<,>^\prime:V^\prime\times W^\prime\rightarrow K$ be two bilinear
non-degenerate forms on vector spaces. For any $f\in
\mathrm{Hom}(V, V^\prime)$, the adjoint operator
$\mathrm{adj}(f)\in \mathrm{Hom}(W^\prime, W)$ of $f$ is defined
to be the one such that $<x,\mathrm{adj}(f)(y)>=<f(x),y>^\prime$
for any $x\in V$ and $y\in W^\prime$. If $(A,m,\eta, \sigma)$ is a
braided algebra, then its adjoint
$(B,\mathrm{adj}(m),\mathrm{adj}(\eta), \mathrm{adj}(\sigma))$ is
a braided coalgebra. A similar statement for braided coalgebras
holds. This indicates some sort of duality between braided
algebras and braided coalgebras.
\end{remark}

The braided algebra and braided coalgebra structures given by
Remark 2.3.1 are trivial. We give nontrivial examples by using
braided vector spaces as follows.

Let $(V,\sigma)$ be a braided vector space. For any $i,j\geq 1$,
we denote\[\chi_{ij}=\left(\begin{array}{cccccccc}
1&2&\cdots&i&i+1&i+2&\cdots & i+j\\
j+1&j+2&\cdots&j+i&1& 2 &\cdots & j
\end{array}\right),\] and define $\beta:T(V)\underline{\otimes} T(V)\rightarrow T(V)\underline{\otimes} T(V)$ by requiring that $\beta_{ij}=T^\sigma_{\chi_{ij}}$ on $V^{\otimes i}\underline{\otimes} V^{\otimes
j}$. For convenience, we denote by $\beta_{0i}$ and $\beta_{i0}$
the usual flip map.

It is easy to see that $\beta$ is a braiding on $T(V)$ and $(T(V),
m, \beta)$ is a braided algebra, where $m$ is the concatenation
product. The algebra $(T(V), m, \beta)$ has a sort of universal
property in the category of braided algebras (see \cite{AMS},
Theorem 1.17 ).

We define $\delta$ to be the deconcatenation on $T(V)$, i.e.,
$$\delta(v_1\otimes\cdots\otimes v_n)=\sum_{i=0}^n(v_1\otimes\cdots\otimes v_i)\underline{\otimes }(v_{i+1}\otimes\cdots\otimes
v_n).$$ We denoted by $T^c(V)$ the coalgebra $(T(V),\delta)$. This
coalgebra is cofree among connected coalgebras. For more
information, one can see \cite{LoR}.

The coalgebra $T^c(V)$ is the dual construction of $(T(V), m)$. So
$(T^c(V),\beta)$ is a braided
 coalgebra.

Now we recall the construction of braided algebras and braided
coalgebras in the category of Yetter-Drinfel'd modules.

Recall that a triple $(V,\cdot, \rho)$ is a \emph{(left)
Yetter-Drinfel'd module} over a Hopf algebra $H$ if $(V,\cdot)$ is
a left $H$-module, $(V,\rho)$ is a left $H$-comodule, and for any
$h\in H$ and $v\in V$, $$\sum h_{(1)}v_{(-1)}\otimes h_{(2)}\cdot
v_{(0)}=\sum (h_{(1)}\cdot v)_{(-1)}h_{(2)}\otimes (h_{(1)}\cdot
v)_{(0)}.
$$
The category of Yetter-Drinfel'd modules over $H$, denoted
$^H_H\mathcal{YD}$, is a braided tensor category (for the
definition, see, e.g., \cite{Ka}). Given two objects $V,W$ in
$^H_H\mathcal{YD}$, the commutativity constraint $c_{V,W}$
associated to $V$ and $W$ is given by $c_{V,W}(v\otimes w)=\sum
v_{(-1)}\cdot w\otimes v_{(0)}$ , for any $ v\in V, w\in W$.

An algebra $(A,m,1)$ is said to be in $^H_H\mathcal{YD}$ if $A$ is
an object in $^H_H\mathcal{YD}$, and the multiplication $m$ and
the unit map are morphisms in $^H_H\mathcal{YD}$. That means
$(A,m,1)$ is both a comodule-algebra and a module-algebra. There
is a dual description of coalgebras. A coalgebra
$(C,\bigtriangleup,\varepsilon)$ is said to be in
$^H_H\mathcal{YD}$ if $C$ is an object in $^H_H\mathcal{YD}$, and
the coproduct $\bigtriangleup$ and the counit $\varepsilon$ are
morphisms in $^H_H\mathcal{YD}$. That means
$(C,\bigtriangleup,\varepsilon)$ is both a comodule-coalgebra and
module-coalgebra. One has the following proposition immediately
(see, e.g., \cite{Ta}).

\begin{proposition}1. If $(A,m,1)$ is an algebra in $^H_H\mathcal{YD}$, then $(A, m,c_{A,A})$ is a braided algebra.

2. If $(C,\bigtriangleup,\varepsilon)$ is a coalgebra in
$^H_H\mathcal{YD}$, then $(C,\bigtriangleup, c_{C,C})$ is a
braided coalgebra.\end{proposition}

Moreover, we have that

\begin{proposition}Let $V$ and $W$ be Yetter-Drinfel'd modules over $H$.

1 If both $V$ and $W$ are module-algebras and comodule-algebras.
Then $(V\otimes W, c_{V\otimes W,V\otimes W})$ is a braided
algebra with the following product: for any $v,v^\prime\in V$ and
$w,w^\prime\in W$,
$$(v\otimes w)\star(v^\prime \otimes w^\prime)=\sum
v(w_{(-1)}\cdot v^\prime)\otimes w_{(0)}w^\prime.$$

2 If both $V$ and $W$ are module-coalgebras and
comodule-coalgebras. Then $(V\otimes W, c_{V\otimes W,V\otimes
W})$ is a braided coalgebra with the following coproduct: for any
$v,v^\prime\in V$,
$$\bigtriangleup(v\otimes
w)=\sum_{(v),(w)}v^{(1)}\otimes(v^{(2)})_{(-1)}\cdot
w^{(1)}\otimes (v^{(2)})_{(0)}\otimes w^{(2)}.$$
Here, for avoiding the ambiguity, we denote $\bigtriangleup(v)=\sum_{(v)}v^{(1)}\otimes v^{(2)}$ and $\bigtriangleup(w)=\sum_{(w)}w^{(1)}\otimes w^{(2)}$.\end{proposition}

The product and coproduct introduced in the above proposition are
the generalizations of smash products and smash coproducts
respectively. This is related to some work of Lambe and Radford
(\cite{LR}, pp. 115-119), but without considering the notion of
braided algebras.

\begin{example} [Woronowicz's braiding]For any Hopf algebra $(H, m, \eta, \bigtriangleup,
\varepsilon,S)$, Woronowicz \cite{Wo2} constructed two braidings
on $H$: for any $a,b\in H$,
\begin{eqnarray*}
T_H(a\otimes b)&=&\sum_{(b)}b_{(2)}\otimes
aS(b_{(1)})b_{(3)},\\
T_H^\prime(a\otimes b)&=&\sum_{(b)}b_{(1)}\otimes
S(b_{(2)})ab_{(3)}.\end{eqnarray*}

We consider $H^{op}=(H, m\circ \tau, \eta, \bigtriangleup,
\varepsilon,S^{-1})$ and $H^{cop}=(H, m, \eta, \tau\circ
\bigtriangleup, \varepsilon,S^{-1})$. Denote
$F_{H}=T^{-1}_{H^{op}}$ and
$F_{H}^\prime=(T^\prime_{H^{cop}})^{-1}$, then
\begin{eqnarray*}F_H(a\otimes b)&=&\sum_{(a)}a_{(1)}S(a_{(3)})b\otimes a_{(2)}
,\\
F_H^\prime(a\otimes b)&=&\sum_{(a)}a_{(1)}b S(a_{(2)})\otimes
a_{(3)}.\end{eqnarray*}

It is well-known that $H$ is a Yetter-Drinfel'd module over itself
with the following structures: for any $x,h\in H$,
\[\left\{\begin{array}{lll}
x\cdot h&=&\sum_{(x)}x_{(1)}h S(x_{(2)}),\\[5pt]
\rho(h)&=&\sum_{(h)}h_{(1)}\otimes h_{(2)}.\\
\end{array} \right.
\]
It is easy to check that $H$ is a module-algebra and a
comodule-algebra with these structures. The braiding from
Yetter-Drinfel'd module structure is just $F^\prime$. So $(H,m,
F^\prime)$ is a braided algebra.

Dually, $H$ has also the following Yetter-Drinfel'd module
structure: for any $x,h\in H$,
\[\left\{\begin{array}{lll}
x\cdot h&=&xh,\\[5pt]
\rho(h)&=&\sum_{(h)}h_{(1)}S(h_{(3)})\otimes h_{(2)}.\\
\end{array} \right.
\]
It is easy to check that $H$ is a module-coalgebra and a
comodule-coalgebra with these structures. The braiding from
Yetter-Drinfel'd module structure is just $F$. So
$(H,\bigtriangleup,F)$ is a braided coalgebra.
\end{example}

In the rest of this section, we focus on the category of
Yetter-Drinfel'd modules over a special kind of Hopf algebras--the
quasi-triangular Hopf algebra (for definition, see \cite{D} or
\cite{Ka}).

Let $(H, \mathcal{R})$ be a quasi-triangular Hopf algebra with
R-matrix $\mathcal{R}=\sum_i s_i\otimes t_i\in H\otimes H$.

For any $H$-module $M$, we define $\rho: M\rightarrow H\otimes M$
by $\rho(m)=\sum_i t_i\otimes s_i\cdot m$. Then $(M,\cdot,\rho)$
is a Yetter-Drinfel'd module over $H$ and the braiding $\sigma_M$
is just the action of the $R$-matrix of $H$ (see, e.g.,
\cite{CMZ}).
\begin{theorem}Under the assumptions above, if $(A, m)$ is a module-algebra over $(H, \mathcal{R})$, then $(A,m,\sigma_A)$ is a braided algebra.\end{theorem}
\begin{proof}We only need to check that $A$ is also a comodule-algebra. Notice that the R-matrix $\mathcal{R}$ satisfies $(\bigtriangleup\otimes
\mathrm{id})(\mathcal{R})=\mathcal{R}_{13}\mathcal{R}_{23}$, i.e.,
$$\sum_i \bigtriangleup (s_i)\otimes t_i=\sum_{k,l}s_k\otimes
s_l\otimes t_k t_l.$$Hence $$\sum_i\sum_{(s_i)} t_i\otimes
(s_i)_{(1)}\otimes (s_i)_{(2)}=\sum_{k,l}t_k t_l\otimes s_k\otimes
s_l.$$For any $a,b\in A$, we have
\begin{eqnarray*}
\sum_{(ab)} (ab)_{(-1)}\otimes (ab)_{(0)}&=&\sum_i t_i\otimes
s_i\cdot (ab)\\
&=&\sum_{i,(s_i)}t_i\otimes \big((s_i)_{(1)}\cdot
a\big)\big((s_i)_{(2)}\cdot
b\big)\\
&=&\sum_{k,l}t_k t_l\otimes (s_k\cdot a)( s_l\cdot b)\\
&=&\sum_{(a),(b)}a_{(-1)}b_{(-1)}\otimes a_{(0)}b_{(0)}.
\end{eqnarray*}
Finally,
\begin{eqnarray*}
\rho(1_A)&=&\sum_i t_i\otimes s_i\cdot 1_A\\
&=&\sum_i\varepsilon(s_i)t_i\otimes 1_A\\
&=&1_H\otimes 1_A,\end{eqnarray*} where the last equality follows
from the fact $(\varepsilon\otimes
\mathrm{id})(\mathcal{R})=1$.\end{proof}

\begin{theorem}With the assumptions above, if $(C,\bigtriangleup)$ is a module-coalgebra over $(H, \mathcal{R})$, then $(C,\bigtriangleup,\sigma_C)$ is a braided coalgebra.\end{theorem}
\begin{proof}It follows from a direct computation in some spirit as the preceding one.\end{proof}

\section{Examples related to quantum groups}
For the relation between quantum groups and braidings, one would
expect there are some examples of braided algebras coming from
quantum groups. In this section, we prove that the upper
triangular part of a quantum group makes sense by using the result
about quantum shuffles in \cite{Ro}.

 For a Yetter-Drinfel'd module $V$ which is
both a module-algebra and a comodule-algebra, $V^{\otimes i}$ is a
braided algebra for each $i$ by Proposition 2.2. One can have
another interesting example of braided algebras as follows, which
will be generalized for any braided vector space later.

We first recall some terminologies. An $(i,j)$-shuffle is an
element $w\in \mathfrak{S}_{i+j}$ such that $w (1) < \cdots <w
(i)$ and $w (i+1) < \cdots <w (i+j)$. We denote by
$\mathfrak{S}_{i,j}$ the set of all $(i,j)$-shuffles.

Let $V$ be a Yetter-Drinfel'd module over a Hopf algebra $H$ with
the natural braiding $\sigma$. In \cite{Ro}, the following
associative product on $T(V)$ was constructed (in fact, the
construction works for any braided vector space, one can see
\cite{FG}): for any $x_1,\ldots, x_{i+j}\in V$,
\begin{equation*}(x_1\otimes\cdots\otimes
x_i)\textrm{{\cyr sh}}_{\sigma}(x_{i+1}\otimes\cdots\otimes
x_{i+j})=\sum_{w\in \mathfrak{S}_{i,j}}T_w(x_1\otimes\cdots\otimes
x_{i+j}).\end{equation*}

The space $T(V)$ equipped with the product $\textrm{{\cyr
sh}}_{\sigma}$ is called the \emph{quantum shuffle algebra} and
denoted by $T_\sigma(V)$. Moreover, the Yetter-Drinfel'd module
$T_\sigma(V)$ is a module-algebra and a comodule-algebra with the
diagonal action and coaction respectively (see \cite{Ro},
Proposition 9). So $T_\sigma(V)$ is a braided algebra. In fact,
the result holds for any braided vector space.

\begin{theorem}Let $(V,\sigma)$ be a braided vector space. Then $(T_\sigma(V),\beta)$ is a braided algebra. The subalgebra
$S_\sigma(V)$ of $T_\sigma(V)$ generated by $V$ is also a braided
algebra with the braiding $\beta$.
\end{theorem}
\begin{proof}For any triple $(i,j,k)$ of positive integers and any $w\in\mathfrak{S}_{i,j}$, we have that $$(1_{\mathfrak{S}_k}\times w)(\chi_{ik}\times 1_{\mathfrak{S}_j} )(1_{\mathfrak{S}_i}\times \chi_{jk})=\chi_{i+j,k}(w\times 1_{\mathfrak{S}_k}).$$ And all the expressions are reduced. It follows that $$(\mathrm{id}_V^{\otimes k}\otimes \textrm{{\cyr sh}}_\sigma)(\beta_{ik}\otimes \mathrm{id}_V^{\otimes j})(\mathrm{id}_V^{\otimes i}\otimes \beta_{jk})=\beta_{i+j,k}(\textrm{{\cyr sh}}_\sigma \otimes \mathrm{id}_V^{\otimes k}).$$The other conditions can be proved
similarly. Hence $(T_\sigma(V),\beta)$ is a braided algebra.

From the definition, $S_\sigma(V)=\oplus_{i\geq
0}\mathrm{Im}(\sum_{w\in \mathfrak{S}_i}T^\sigma_w)$. By observing
that $\chi_{ij}(w\times w^\prime)=(w^\prime\times w)\chi_{ij}$ for
any $w\in \mathfrak{S}_i$ and $w^\prime\in \mathfrak{S}_j$ and all
the expressions are reduced, we have that $\beta$ is a braiding on
$S_\sigma(V)$. It is certainly a braided algebra since it is a
subalgebra of $T_\sigma(V)$.
\end{proof}

\begin{remark}By using the dual construction, we know $(T(V), \beta)$ is a braided coalgebra with the following coproduct $\bigtriangleup$: for any $x_1,\ldots, x_{n}\in V$, the component of $\bigtriangleup(x_1\otimes\cdots\otimes
x_n)$ in $V^{\otimes p}\underline{\otimes} V^{\otimes n-p}$ is
\begin{eqnarray*}\bigtriangleup(x_1\otimes\cdots\otimes
x_n)=\sum_{w\in
\mathfrak{S}_{p,n-p}}T_{w^{-1}}(x_1\otimes\cdots\otimes
x_{n}).\end{eqnarray*}
\end{remark}

\begin{example}[Quantum exterior algebras] Let $V$ be a vector
space over $\mathbb{C}$ with basis $\{e_1,\ldots, e_N\}$. Take a
nonzero scalar $q\in \mathbb{C}$. We define a braiding $\sigma$ on
$V$ by
\[\sigma(e_{i}\otimes e_{j})=\left\{
\begin{array}{lll}
e_{i}\otimes e_{i},&& i=j,\\
q^{-1}e_{j}\otimes e_{i},&& i<j,\\
q^{-1}e_{j}\otimes e_{i}+(1-q^{-2})e_{i}\otimes e_{j},&&i>j.
\end{array} \right.
\]
Then $\sigma$ satisfies the Iwahori's quadratic equation
$(\sigma-\mathrm{id}_{V\otimes
V})(\sigma+q^{-2}\mathrm{id}_{V\otimes V})=0.$ In fact, this
$\sigma$ is given by the action of the $R$-matrix on the
fundamental representation of $U_q\mathfrak{sl}_N$. By a result of
Gurevich (see \cite{Gu}, Proposition 2.13), we know that
$T(V)/I\cong\oplus_{i\geq 0}\mathrm{Im}(\sum_{w\in
\mathfrak{S}_i}(-1)^{l(w)}T_w)$ as algebras, where $l(w)$ is the
length of $w$ and $I$ is the ideal of $T(V)$ generated by
$\mathrm{Ker}(\mathrm{id}_{V\otimes V}-\sigma)$. By easy
computation, we get that $\mathrm{Ker}(\mathrm{id}_{V\otimes
V}-\sigma)=\mathrm{Span}_{\mathbb{C}}\{ e_{i}\otimes
e_{i},q^{-1}e_{i}\otimes e_{j}+e_{j}\otimes e_{i}(i<j)\}$. We
denote by $e_{i_{1}}\wedge \cdots \wedge e_{i_{s}}$ the image of
$e_{i_{1}}\otimes \cdots \otimes e_{i_{s}}$ in $S_\sigma(V)$. So
$S_\sigma(V)$ is an algebra generated by $(e_i)$ with the
relations $e_i^2=0$ and $e_j\wedge e_i=-q^{-1}e_i\wedge e_j$ if
$i<j$. This $S_\sigma(V)$ is called the \emph{quantum exterior
algebra} over $V$. It is a finite dimensional braided algebra with
the braiding $\beta$.

The quantum exterior algebra has another braided algebra structure
as follows. We denote the increasing set $(i_{1}, \ldots ,i_{s})$
by $\underline{i}$ and so on. For $1\leq i_{1}<\cdots <i_{s}\leq
N$ and $1\leq j_{1}<\cdots <j_{t}\leq N$, we denote
\[(i_{1},\cdots ,i_{s}|j_{1},\cdots ,j_{t})=\left \{
\begin{array}{lll}
0,&& \mathrm{if}\ \underline{i}\cap \underline{j}\neq \emptyset,\\
2\sharp \{(i_{k},j_{l})|i_{k}>j_{l}\}-st,&&\mathrm{otherwise}.
\end{array} \right.
\]

Using the above notation, it is easy to see that $$e_{i_{1}}\wedge
\cdots \wedge e_{i_{s}}\wedge e_{j_{1}}\wedge \cdots \wedge
e_{j_{t}}=(-q)^{-(i_{1},\cdots ,i_{s}|j_{1},\cdots ,j_{t})}
e_{j_{1}}\wedge \cdots \wedge e_{j_{t}}\wedge e_{i_{1}}\wedge
\cdots \wedge e_{i_{s}}.$$

We define the \emph{q-flip} $\mathscr{T}=\bigoplus_{s,t}
\mathscr{T}_{s,t}$: $S_\sigma(V)\otimes S_\sigma(V)\rightarrow
S_\sigma(V)\otimes S_\sigma(V)$ as follows: for $1\leq
i_{1}<\cdots <i_{s}\leq N$ and $1\leq j_{1}<\cdots <j_{t}\leq N$,
$$\mathscr{T}_{s,t}(e_{i_{1}}\wedge \cdots \wedge e_{i_{s}}\otimes e_{j_{1}}\wedge \cdots \wedge e_{j_{t}})=(-q)^{(i_{1},\cdots ,i_{s}|j_{1},\cdots ,j_{t})} e_{j_{1}}\wedge \cdots \wedge e_{j_{t}}\otimes e_{i_{1}}\wedge \cdots \wedge e_{i_{s}}.
$$

Obviously, $\mathscr{T}$ is a braiding and it induces a
representation of the symmetric group since
$\mathscr{T}^{2}=\mathrm{id}$. Furthermore, it is easy to show
that $( S_\sigma(V), \wedge, \mathscr{T})$ is a braided algebra
and $( S_\sigma(V), \delta, \mathscr{T})$ is a braided coalgebra.
\end{example}

Originally, quantum shuffle algebras were discovered from the
cotensor algebras (see \cite{Ro}). Cotensor algebras are the dual
construction of tensor algebras. They are constructed over Hopf
bimodules.

\begin{definition}[ \cite{N}, \cite{Wo1}]Let $H$ be a Hopf algebra. A \emph{Hopf bimodule} over $H$ is a vector space $M$ given with an $H$-bimodule structure, an $H$-bicomodule structure with left and right coactions $\delta_L: M\rightarrow H\otimes M$, $\delta_R: M\rightarrow M\otimes H$ which commute in the following sense: $(\delta_L\otimes \mathrm{id}_M)\delta_R=(\mathrm{id}_M\otimes \delta_R)\delta_L$, and such that $\delta_L$ and $\delta_R$ are morphisms of $H$-bimodules.
\end{definition}

We denote by $M^R$ the subspace of right coinvariants, i.e.,
$M^R=\{m\in M|\delta_R(m)=m\otimes 1\}$. Then $M^R$ is a left
Yetter-Drinfel'd module with coaction $\delta$ and the left
adjoint action given by: for any $h\in H$ and $m\in M^R$,
$$h\cdot m=\sum h_{(1)}mS(h_{(2)}).$$

Combining the discussions in the preceding section, it is not hard
to see that the cotensor algebra is both a braided algebra and a
braided coalgebra. Here, we give a more general description of
this phenomenon in the framework due to Radford \cite{Ra} of
bialgebras with a projection onto a Hopf algebra. We first recall
some results in \cite{Ra} which we will use in our discussion.

Let $H$ be a Hopf algebra with antipode $S$ and $A$ be a
bialgebra. Suppose there are two bialgebra maps $i: H\rightarrow
A$ and $\pi: A\rightarrow H$ such that $\pi\circ i=\mathrm{id}_H$.
Set $\Pi=\mathrm{id}_A\star(i\circ S\circ \pi)$, where $\star$ is
the convolution product on $\mathrm{End}(A)$, and $B=\Pi(A)$.

1. The bialgebra $A$ is a Hopf bimodule over $H$ with actions
$h\cdot a=i(h)a$ and $a\cdot h=ai(h)$, coactions $\delta_L(a)=\sum
\pi(a_{(1)})\otimes a_{(2)}$ and $\delta_R(a)=\sum a_{(1)}\otimes
\pi(a_{(2)})$ for any $h\in H$ and $a\in A$. Obviously, by the
projection formula from a Hopf bimodule to its right coinvariant
subspace, $A^R=B$. So $B$ is a left Yetter-Drinfel'd module over
$H$ with the left adjoint action.

2. The set $B$ is a subalgebra of A. Furthermore it is both a
module-algebra and a comodule-algebra. Moreover, $B$ has a
coalgebra structure such that $\Pi$ is a coalgebra map. With this
coalgebra structure, $B$ is both a module-coalgebra and a
comodule-coalgebra.

3. The map $B\otimes H\rightarrow A$ given by $b\otimes h\mapsto
bi(h)$ is a bialgebra isomorphism, where $B\otimes H$ is with the
smash product and smash coproduct.

So by combining Woronowicz's examples on $H$ and Proposition 2.5
for tensor products, the bialgebra $A$ is both a braided algebra
and a braided coalgebra. If $A$ is moreover a Hopf algebra, then
it is again a braided algebra and braided coalgebra using directly
Woronowicz's braidings. Obviously, these two braided algebra
(resp. coalgebra) structures are different.

Now we restrict our attention on cotensor algebras, which will
give us braided algebras related to quantum groups. For a Hopf
bimodule $M$ over $H$, one can construct the cotensor algebra
$T^c_H(M)$ over $H$ and $M$. More precisely, we define $M\square
M=\mathrm{Ker}(\delta_R\otimes \mathrm{id}_M-\mathrm{id}_M\otimes
\delta_L)$ and $M^{\square k}= M^{\square k-1}\square M$ for
$k\geq 3$. And the cotensor algebra built over $H$ and $M$ is
$T^c_H(M)=H\oplus M\oplus\oplus_{k\geq 2}M^{\square k}$. It is
again a Hopf bimodule over $H$. From the universal property of
cotensor algebras, one can construct a Hopf algebra structure with
a complicated multiplication on $T^c_H(M)$. We denote by $S_H(M)$
the subalgebra of $T^c_H(M)$ generated by $H$ and $M$. Then
$S_H(M)$ is a sub-Hopf algebra. For more details, one can see
\cite{N1}. Apparently, the cofree Hopf algebra $T^c(V)$ defined in
Section 2 is the cotensor algebra over the trivial Hopf algebra
$K$ and the trivial Hopf bimodule $V$. Here $V$ is a Hopf bimodule
with scalar multiplication and the coactions defined by
$\delta_L(v)=1\otimes v$ and $\delta_R(v)=v\otimes 1$ for any
$v\in V$.

Since the inclusion $H\rightarrow T^c_H(M)$ and the projection
$T^c_H(M)\rightarrow H$ are bialgebra maps, we get:

\begin{theorem}Let $M$ be a Hopf bimodule over $H$. Then $T^c_H(M)$ is both a braided algebra and a braided
coalgebra. So is $S_H(M)$.
\end{theorem}

As an application of the above theorem, we consider the following
special case. Let $G=\mathbb{Z}^r\times \mathbb{Z}/l_1
\times\mathbb{Z}/l_2 \times\cdots \mathbb{Z}/l_p$ and $H=K[G]$ be
the group algebra of $G$. We fix generators $K_1,\ldots, K_N$ of
$G$ ($N=r+p$). Let $V$ be a vector space over $\mathbb{C}$ with
basis $\{e_1,\ldots,e_N\}$. It is known that $V$ is a
Yetter-Drinfel'd module over $H$ with action and coaction given by
$K_i\cdot e_j=q_{ij}e_j$ and $\delta_L(e_i)=K_i\otimes e_i $ with
some nonzero scalar $q_{ij}\in \mathbb{C}$ respectively. The
braiding coming from the Yetter-Drinfel'd module structure is
given by $\sigma(e_i\otimes e_j)=q_{ij}e_j\otimes e_i$. Now we
choose special $q_{ij}$ to construct meaningful examples. Let
$A=(a_{ij})_{1\leq i,j\leq N}$ be a symmetrizable generalized
Cartan matrix, $(d_1,\ldots, d_N)$ be positive relatively prime
integers such that $(d_ia_{ij})$ is symmetric. Let $q\in
\mathbb{C}$ and define $q_{ij}=q^{d_i a_{ij}}$. By Theorem 15 in
\cite{Ro}, $S_H(M)$ is isomorphic, as a Hopf algebra, to the sub
Hopf algebra $U_q^+$ of the quantized universal enveloping algebra
associated with $A$ when $G=\mathbb{Z}^N$ and $q$ is not a root of
unity; $S_H(M)$ is isomorphic, as a Hopf algebra, to the quotient
of the restricted quantized enveloping algebra $u_q^+$ by the
two-sided Hopf ideal generated by the elements $(K_i^l-1)$,
$i=1,\ldots, N$ when $G=(\mathbb{Z}/l)^N$ and $q$ is a primitive
$l$-th root of unity. Then we have:
\begin{corollary}Both $U_q^+$ and $u_q^+$ are braided algebras and braided coalgebras.
\end{corollary}

We use the above special $S_\sigma(V)\otimes H$ to illustrate the
difference between the braiding coming from Woronowicz's
construction and the one from the tensor product of two
Yetter-Drinfel'd modules.

We use the following notation: for any $g=K_1^{i_1}\cdots
K_N^{i_N}\in G$, $q_{gj}=q_{1j}^{i_1}\cdots q_{Nj}^{i_N}$, i.e.,
$g\cdot e_j=q_{gj}e_j$. For any $g,h\in G$, Woronowicz's braiding
$F^\prime$ has the following action on $S_\sigma(V)\otimes
 H$:
\begin{eqnarray*}
\lefteqn{F^\prime\Big((e_i\otimes g)\otimes(e_j\otimes
h)\Big)}\\[3pt]
&=&q_{ij}q_{gj}(e_j\otimes h)\otimes (e_i\otimes
g)-q_{ij}q_{gj}q_{hi}(e_je_i\otimes h)\otimes
g+q_{gj}(e_ie_j\otimes h)\otimes g.
\end{eqnarray*}
But the braiding in the category of Yetter-Drinfel'd modules is :
\begin{eqnarray*} \Sigma\Big((e_i\otimes
g)\otimes(e_j\otimes h)\Big)&=&q_{ij}(e_j\otimes
h)\otimes(e_i\otimes g).
\end{eqnarray*}

\section{Quantum multi-brace algebras}
In this section, we introduce and study the main objects of this
paper: quantum multi-brace algebras. The fact that they lead
naturally to braided algebras relies on compatibilities between
the braiding and the maps $M_{ij}$ involved. Part of our task is
to deduce from our assumptions in the definition all the
identities satisfied by braiding, coproducts and maps $M_{ij}$,
which is done in a series of lemmas.

Let $(C,\bigtriangleup, \varepsilon)$ be a coalgebra with a
preferred group-like element $1_C\in C$ and denote
$\overline{\bigtriangleup}(x)=\bigtriangleup(x)-x\otimes
1_C-1_C\otimes x$ for any $x\in C$. The map
$\overline{\bigtriangleup}$ is called the \emph{reduced
coproduct}. It is coassociative. The following definition and
universal property play an essential role in the theory of quantum
multi-brace algebras.

\begin{definition}[\cite{Q}]A coalgebra $(C,\bigtriangleup)$ with a
preferred group-like element $1_C\in C$ is said to be
\emph{connected} if $C=\cup_{r\geq 0}F_r C$,
where\begin{eqnarray*}F_0C&=&K1_C,\\
F_rC&=&\{x\in C|\overline{\bigtriangleup}(x)\in F_{r-1}C\otimes
F_{r-1}C \},\ \ \mathrm{for\ r\geq
1}.\end{eqnarray*}\end{definition}

There is a well-known universal property for the cofree Hopf
algebra $T^c(V)$ in the category of connected coalgebras (see,
e.g., \cite{LoR}):

\begin{proposition}Given a connected coalgebra $(C,\bigtriangleup,\varepsilon)$ and a linear map $\phi:C\rightarrow V$ such that $\phi(1_C)=0$, there is a unique coalgebra morphism $\overline{\phi}:C\rightarrow T^c(V)$ which extends $\phi$, i.e., $P_V\circ \overline{\phi}=\phi$, where $P_V:T^c(V)\rightarrow V $ is the projection onto $V$. Explicitly, $\overline{\phi}=\varepsilon+\sum_{n\geq 1}\phi^{\otimes n}\circ\overline{\bigtriangleup}^{(n-1)}$.\end{proposition}

Indeed, the sum $\sum_{n\geq 1}\phi^{\otimes
n}\circ\overline{\bigtriangleup}^{(n-1)}$ in the above formula for
the map $\overline{\phi}$ is finite since $C$ is connected and
$\phi(F_0 C)=0$ imply that $\phi^{\otimes
n}\circ\overline{\bigtriangleup}^{(n-1)}$ vanishes on $F_{n-1} C$.
There is a useful consequence of this universal property.

\begin{corollary}Let $C$ be a connected coalgebra. If $\Phi,\Psi:C\rightarrow T^c(V)$ are coalgebra maps such that $P_V\circ \Phi=P_V\circ \Psi$ and $P_V\circ\Phi(1_C)=0=P_V\circ\Psi(1_C)$, then $\Phi=\Psi$. \end{corollary}

Using Proposition 2.2 and the fact $(T^c(V),\beta)$ is a braided
coalgebra, we know there is a coalgebra structure on
$T^c(V)^{\otimes i}$ by combining $\beta$ and $\delta$:
$$\bigtriangleup_{\beta,i}=T^\beta_{w_i^{-1}}\circ \delta^{\otimes i},$$
and the counit is $\varepsilon^{\otimes i}$.

\begin{proposition}Let $(V,\sigma)$ be a braided vector space. Then for any $n\geq 1$, the coalgebra $(T^c(V)^{\otimes n}, \bigtriangleup_{\beta,n})$ is connected.\end{proposition}
\begin{proof}Obviously, $1^{\underline{\otimes} n}$ is a group-like element of $T^c(V)^{\otimes n} $. For any $r\geq 0$, we have that $$F_r=F_r(T^c(V)^{\otimes n} )=\bigoplus_{0\leq i_1+\cdots+i_n\leq r}V^{\otimes i_1}\underline{\otimes}\cdots\underline{\otimes}V^{\otimes i_n}.$$
\end{proof}
From now on, we use $\bigtriangleup_{\beta}$ to denote
$\bigtriangleup_{\beta,2}$ when $n=2$. Since $w_2^{-1}=s_2\in
\mathfrak{S}_4$,
$\bigtriangleup_\beta=(\mathrm{id}_{T^c(V)}\otimes \beta\otimes
\mathrm{id}_{T^c(V)})\circ (\delta\otimes\delta)$.

Let $M=\oplus M_{pq}: T^c(V)\underline{\otimes} T^c(V)\rightarrow
V$ be a linear map such that $M_{pq}:V^{\otimes
p}\underline{\otimes }V^{\otimes q}\rightarrow V$, and \[\left\{
\begin{array}{lllll}
M_{00}&=&0,&& \\
M_{10}&=&\mathrm{id}_V&=&M_{01},\\
M_{n0}&=&0&=&M_{0n},\ \mathrm{for}\ n\geq 2.
\end{array} \right.
\]

Since $M(1\underline{\otimes} 1)=0$, there is a unique coalgebra
map $\ast: T^c(V)\underline{\otimes} T^c(V)\rightarrow T^c(V)$ by
the universal property of $T^c(V)$. Explicitly,
$$\ast=(\varepsilon\otimes \varepsilon)+\sum_{n\geq 1}M^{\otimes
n}\circ \overline{\bigtriangleup_\beta}^{(n-1)}.$$ We shall
investigate conditions under which $\ast$ is an associative
product. Here we start by giving another form of $\ast$ by using
the map $M$ and the deconcatenation $\delta$.

\begin{proposition}For $n\geq 0$, we have that \begin{eqnarray*}\bigtriangleup_\beta^{(n)}=T_{w_{n+1}}^\beta\circ (\delta^{(n)})^{\otimes 2}.\end{eqnarray*} \end{proposition}
\begin{proof}We use induction on $n$.

When $n=0$, it is trivial since $w_1=1_{\mathfrak{S}_2}$.

When $n=1$,
$\bigtriangleup_\beta^{(1)}=\bigtriangleup_\beta=\beta_2(\delta\otimes\delta)=T_{w_{2}}^\beta\circ
(\delta^{(1)})^{\otimes 2}$ since $w_2=s_2$.

When $n=2$, \begin{eqnarray*}
\bigtriangleup_\beta^{(2)}&=&(\bigtriangleup_\beta\otimes
\mathrm{id}_{T^c(V)}\otimes
\mathrm{id}_{T^c(V)})\bigtriangleup_\beta\\[3pt]
&=&\beta_2(\delta\otimes\delta\otimes \mathrm{id}_{T^c(V)}\otimes \mathrm{id}_{T^c(V)})\beta_2(\delta\otimes\delta)\\[3pt]
&=&\beta_2\big(\delta\otimes(\delta\otimes \mathrm{id}_{T^c(V)})\beta\otimes \mathrm{id}_{T^c(V)}\big)\circ(\delta\otimes\delta)\\[3pt]
&=&\beta_2\big(\delta\otimes\beta_2\beta_1(\mathrm{id}_{T^c(V)}\otimes \delta)\otimes \mathrm{id}_{T^c(V)}\big)\circ(\delta\otimes\delta)\\[3pt]
&=&\beta_2\beta_4\beta_3(\delta\otimes \mathrm{id}_{T^c(V)}\otimes \delta\otimes \mathrm{id}_{T^c(V)})\circ(\delta\otimes\delta)\\[3pt]
&=&T_{w_{3}}^\beta\circ (\delta^{(2)})^{\otimes 2}.
\end{eqnarray*}
For $n\geq 3$, \begin{eqnarray*}
\bigtriangleup_\beta^{(n+1)}&=&(\bigtriangleup_\beta\otimes \mathrm{id}_{T^c(V)}^{\otimes 2n})\bigtriangleup_\beta^{(n)}\\[3pt]
&=&\beta_2(\delta\otimes\delta\otimes \mathrm{id}_{T^c(V)}^{\otimes 2n})T_{w_{n+1}}^\beta\circ (\delta^{(n)})^{\otimes 2}\\[3pt]
&=&\beta_2(\delta\otimes\delta\otimes \mathrm{id}_{T^c(V)}^{\otimes 2n})(\mathrm{id}_{T^c(V)}^{\otimes 2}\otimes T_{w_{n}}^\beta)\beta_1\cdots\beta_{n+1}\circ (\delta^{(n)})^{\otimes 2}\\[3pt]
&=&\beta_2(\mathrm{id}_{T^c(V)}^{\otimes 2}\otimes T_{w_{n}}^\beta)(\delta\otimes\delta\otimes \mathrm{id}_{T^c(V)}^{\otimes 2n})\beta_1\cdots\beta_{n+1}\circ (\delta^{(n)})^{\otimes 2}\\[3pt]
&=&\beta_2(\mathrm{id}_{T^c(V)}^{\otimes 2}\otimes T_{w_{n}}^\beta)\beta_4\beta_3\beta_5\beta_4\cdots\beta_{n+3}\beta_{n+2}\\[3pt]
&&\circ(\delta\otimes\otimes \mathrm{id}_{T^c(V)}^{\otimes n}\otimes\delta\otimes \mathrm{id}_{T^c(V)}^{\otimes n})\circ (\delta^{(n)})^{\otimes 2}\\[3pt]
&=&T_{w_{n+2}}^\beta\circ (\delta^{(n+1)})^{\otimes 2}.
\end{eqnarray*}The third and last equalities follow from the fact that $w_{n+1}=(1_{\mathfrak{S}_2}\times w_n)s_2\cdots s_{n+1}$ for $n\geq 1$, $w_{n+2}=s_2(1_{\mathfrak{S}_4}\times w_n)s_4 s_3 s_5 s_4\cdots s_{n+3}s_{n+2}$ for $n\geq 3$ and both expressions are reduced.\end{proof}

\begin{lemma}For $n\geq 1$, we have $M^{\otimes n}\bigtriangleup_\beta^{(n-1)}(1\underline{\otimes }1)=0$. \end{lemma}
\begin{proof}It follows from the fact that $\bigtriangleup_\beta^{(n-1)}(1\underline{\otimes }1)=(1\underline{\otimes }1)^{\otimes n}$ and $M_{00}=0$.\end{proof}

\begin{proposition}For $n\geq 1$, we have $M^{\otimes n}\overline{\bigtriangleup_\beta}^{(n-1)}=M^{\otimes n}\bigtriangleup_\beta^{(n-1)}$\end{proposition}
\begin{proof}We use induction on $n$.

When $n=1$, it is trivial.

For $n\geq 2$ any $u,v\in T^c(V)$,\begin{eqnarray*}
\lefteqn{M^{\otimes n}\overline{\bigtriangleup_\beta}^{(n-1)}}\\
&=&\big((M^{\otimes n-1}\overline{\bigtriangleup_\beta}^{(n-2)})\otimes M\big)\overline{\bigtriangleup_\beta}(u\underline{\otimes }v)\\
&=&\big((M^{\otimes n-1}\bigtriangleup_\beta^{(n-2)})\otimes M\big)\Big(\bigtriangleup_\beta(u\underline{\otimes }v)-(1\underline{\otimes }1)\underline{\otimes}(u\underline{\otimes }v)-(u\underline{\otimes }v)\underline{\otimes}(1\underline{\otimes }1)\Big)\\
&=&M^{\otimes n}\bigtriangleup_\beta^{(n-1)}(u\underline{\otimes }v)-\big(M^{\otimes n-1}\bigtriangleup_\beta^{(n-2)}(1\underline{\otimes }1)\big)\underline{\otimes}M_{11}(u\underline{\otimes }v)\\
&&-\big(M^{\otimes n-1}\bigtriangleup_\beta^{(n-2)}(u\underline{\otimes }v)\big)\underline{\otimes}M_{00}(1\underline{\otimes }1)\\
&=&M^{\otimes n}\bigtriangleup_\beta^{(n-1)}(u\underline{\otimes
}v).
\end{eqnarray*}\end{proof}

From this lemma, the map $\ast$ defined by $M_{pq}$'s can be
rewritten as $\ast=\varepsilon\otimes \varepsilon+\sum_{r\geq
1}M^{\otimes r}\circ \bigtriangleup_\beta^{(r-1)}$. And we have
the following formula immediately.

\begin{corollary}We can rewrite $\ast$ as  \begin{eqnarray*}\ast=\varepsilon\otimes\varepsilon+\sum_{n\geq 1}M^{\otimes n}\circ T_{w_{n}}^\beta\circ (\delta^{(n-1)})^{\otimes 2}.\end{eqnarray*}\end{corollary}

But this $\ast$ is not an associative product on $T^c(V)$ in
general. Now we will generalize the notion of braided algebras by
giving some compatibility conditions between $M_{pq}$'s and the
braiding, and prove that under these conditions the new object
makes $\ast$ to be associative automatically and $T^c(V)$ becomes
a braided algebra with $\ast$.

\begin{definition}A \emph{quantum multi-brace algebra} $(V,M,\sigma)$ is a braided vector space $(V,\sigma)$ equipped with a operation $M=\oplus M_{pq}$, where $$M_{pq}:V^{\otimes p}\otimes V^{\otimes q}\rightarrow
V,\ \ p\geq 0,\ q\geq 0,$$satisfying

1.
\[\left\{
\begin{array}{lllll}
M_{00}&=&0,&& \\
M_{10}&=&\mathrm{id}_V&=&M_{01},\\
M_{n0}&=&0&=&M_{0n},\ \mathrm{for}\ n\geq 2,
\end{array} \right.
\]

2. braid condition: for any $i,j,k\geq 1$,
\[\left\{
\begin{array}{lll}
\beta_{1k}(M_{ij}\otimes \mathrm{id}_V^{\otimes k})&=&(
\mathrm{id}_V^{\otimes k}\otimes M_{ij})\beta_{i+j,k} , \\[3pt]
\beta_{i1}(\mathrm{id}_V^{\otimes i}\otimes
M_{jk})&=&(M_{jk}\otimes \mathrm{id}_V^{\otimes i} )\beta_{i,j+k},
\end{array} \right.
\]

3. associativity condition: for any triple $(i,j,k)$ of positive
integers,
\begin{eqnarray*}
\lefteqn{\sum_{r=1}^{i+j}M_{rk}\circ \big((M^{\otimes r}\circ
\bigtriangleup_\beta^{(r-1)})\otimes
\mathrm{id}_V^{\otimes k}\big)}\\
&=&\sum_{l=1}^{j+k}M_{il}\circ \big(\mathrm{id}_V^{\otimes
i}\otimes( M^{\otimes l}\circ \bigtriangleup_\beta^{(l-1)})\big).
\end{eqnarray*}
\end{definition}
\begin{remark}For any vector space $V$, $(V,\tau)$ is always a braided vector space with the usual flip $\tau$. In this case, the braid condition in the above definition holds automatically, and the quantum multi-brace algebra returns to the classical $\textbf{B}_\infty$-algebra (for the definition of $\textbf{B}_\infty$-algebras, one can see \cite{LoR}).\end{remark}

\begin{example} 1. A braided vector space $(V,\sigma)$ is a
quantum multi-brace algebra with $M_{ij}=0$ except for the pairs
$(1,0)$ and $(0,1)$.

2. A braided algebra $(A,m,\sigma)$ is a quantum multi-brace
algebra with $M_{11}=m$ and $M_{ij}=0$ except for the pairs
$(1,0)$, $(0,1)$ and $(1,1)$.
\end{example}
In the following, we adopt the notation $M_{(i_1,j_1,\ldots,
i_k,j_k)}=M_{i_1 j_1}\otimes\cdots\otimes M_{i_k j_k}$.

\begin{lemma}Let $(V,M,\sigma)$ be a quantum multi-brace algebra. Then for any $k,l\geq 1$, we have\[\left\{
\begin{array}{lll}
\beta_{kl}(M_{(i_1,j_1,\ldots, i_k,j_k)}\otimes
\mathrm{id}_V^{\otimes l})&=&(\mathrm{id}_V^{\otimes l}\otimes
M_{(i_1,j_1,\ldots,
i_k,j_k)})\beta_{i_1+j_1+\cdots+i_k+j_k,l} ,\\[5pt]
\beta_{lk}(\mathrm{id}_V^{\otimes l}\otimes M_{(i_1,j_1,\ldots,
i_k,j_k)} )&=&( M_{(i_1,j_1,\ldots,
i_k,j_k)}\otimes\mathrm{id}_V^{\otimes
l})\beta_{l,i_1+j_1+\cdots+i_k+j_k}.
\end{array} \right.
\] \end{lemma}
\begin{proof}We use induction on $k$.

The case $k=1$ is trivial.
\begin{eqnarray*}
\lefteqn{\beta_{k+1,l}(M_{(i_1,j_1,\ldots,
i_{k+1},j_{k+1})}\otimes \mathrm{id}_V^{\otimes
l})}\\[3pt]
&=&(\beta_{kl}\otimes \mathrm{id}_V)(\mathrm{id}_V^{\otimes
k}\otimes \beta_{1l})(M_{(i_1,j_1,\ldots, i_{k+1},j_{k+1})}\otimes
\mathrm{id}_V^{\otimes l})\\[3pt]
&=&(\beta_{kl}\otimes \mathrm{id}_V)\Big(M_{(i_1,j_1,\ldots,
i_{k},j_{k})}\otimes \beta_{1l}(M_{i_{k+1}j_{k+1}}\otimes
\mathrm{id}_V^{\otimes l})\Big)\\[3pt]
&=&(\beta_{kl}\otimes \mathrm{id}_V)\Big(M_{(i_1,j_1,\ldots,
i_{k},j_{k})}\otimes (\mathrm{id}_V^{\otimes l}\otimes
M_{i_{k+1}j_{k+1}}
)\beta_{i_{k+1}+j_{k+1},l}\Big)\\[3pt]
&=&\Big(\beta_{kl}(M_{(i_1,j_1,\ldots, i_{k},j_{k})}\otimes
\mathrm{id}_V^{\otimes l})\otimes
\mathrm{id}_V\Big)\\[3pt]
&&\circ(\mathrm{id}_V^{\otimes i_1+\cdots+i_k+l}\otimes
M_{i_{k+1}j_{k+1}} )(\mathrm{id}_V^{\otimes
i_1+\cdots+i_k}\otimes\beta_{i_{k+1}+j_{k+1},l})\\[3pt]
&=&\Big((\mathrm{id}_V^{\otimes l}\otimes M_{(i_1,j_1,\ldots,
i_k,j_k)})\beta_{i_1+j_1+\cdots+i_k+j_k,l} \otimes
\mathrm{id}_V\Big)\\[3pt]
&&\circ(\mathrm{id}_V^{\otimes i_1+\cdots+i_k+l}\otimes
M_{i_{k+1}j_{k+1}}
)(\mathrm{id}_V^{\otimes i_1+\cdots+i_k}\otimes\beta_{i_{k+1}+j_{k+1},l})\\[3pt]
&=&(\mathrm{id}_V^{\otimes l}\otimes M_{(i_1,j_1,\ldots,
i_{k+1},j_{k+1})})\\[3pt]
&&\circ(\beta_{i_1+j_1+\cdots+i_k+j_k,l} \otimes
\mathrm{id}_V)(\mathrm{id}_V^{\otimes i_1+\cdots+i_k}\otimes\beta_{i_{k+1}+j_{k+1},l})\\[3pt]
&=&(\mathrm{id}_V^{\otimes l}\otimes M_{(i_1,j_1,\ldots,
i_k,j_k)})\beta_{i_1+j_1+\cdots+i_k+j_k,l}.\\
\end{eqnarray*}The another equality is proved similarly.
\end{proof}

The following notation is adopted to simplify the identities. We
denote by $\bigtriangleup_{\beta\ (i_1,j_1, i_2,j_2)}$ the
composition of $\bigtriangleup_\beta:V^{\otimes
i_1+i_2}\underline{\otimes}V^{\otimes j_1+j_2}\rightarrow
\big(T(V)\underline{\otimes}T(V)\big)\underline{\otimes}\big(T(V)\underline{\otimes}T(V)\big)$
with the projection
$\big(T(V)\underline{\otimes}T(V)\big)\underline{\otimes}\big(T(V)\underline{\otimes}T(V)\big)\rightarrow
(V^{\otimes i_1}\underline{\otimes}V^{\otimes
j_1})\underline{\otimes}(V^{\otimes
i_2}\underline{\otimes}V^{\otimes j_2})$, and by
$$\bigtriangleup_{\beta\ (i_1,j_1,\ldots, i_k,j_k)}^{(k-1)}=(\bigtriangleup_{\beta\ (i_1,j_1,
i_2,j_2)}\otimes \mathrm{id}_V^{\otimes
i_3+j_3+\cdots+i_k+j_k})\circ\bigtriangleup_{\beta\
(i_1+i_2,j_1+j_2,i_3,j_3,\ldots, i_k,j_k)}^{(k-2)}$$ the map from
$V^{\otimes i_1+\cdots+i_k}\underline{\otimes}V^{\otimes
j_1+\cdots+j_k}$ to $V^{\otimes i_1}\underline{\otimes}V^{\otimes
j_1}\underline{\otimes}\cdots\underline{\otimes}V^{\otimes
i_k}\underline{\otimes}V^{\otimes j_k}$ inductively.

\begin{lemma} For any $k,l\geq 1$, we have\[\left\{
\begin{array}{lll}
\lefteqn{\beta_{i_1+j_1+\cdots+i_k+j_k,l}(\bigtriangleup_{\beta\
(i_1,j_1,\ldots, i_k,j_k)}^{(k-1)}\otimes \mathrm{id}_V^{\otimes
l})}\\[5pt]
&=&(\mathrm{id}_V^{\otimes l}\otimes \bigtriangleup_{\beta\
(i_1,j_1,\ldots,
i_k,j_k)}^{(k-1)})\beta_{i_1+j_1+\cdots+i_k+j_k,l} ,\\[5pt]
\lefteqn{\beta_{l,i_1+j_1+\cdots+i_k+j_k}(\mathrm{id}_V^{\otimes
l}\otimes \bigtriangleup_{\beta\ (i_1,j_1,\ldots,
i_k,j_k)}^{(k-1)})}\\[5pt]
&=&( \bigtriangleup_{\beta\ (i_1,j_1,\ldots,
i_k,j_k)}^{(k-1)}\otimes \mathrm{id}_V^{\otimes
l})\beta_{l,i_1+j_1+\cdots+i_k+j_k} .
\end{array} \right.
\]\end{lemma}
\begin{proof}Since $(T^c(V)^{\otimes 2},\bigtriangleup_\beta,T^\beta_{\chi_{22}})$ is a braided coalgebra, we have \[\left\{
\begin{array}{lll}
\lefteqn{(\mathrm{id}_{T^c(V)^{\otimes 2}}\otimes
\bigtriangleup_\beta )T^\beta_{\chi_{22}}}\\[3pt]
&=&(T^\beta_{\chi_{22}}\otimes \mathrm{id}_{T^c(V)^{\otimes 2}})(\mathrm{id}_{T^c(V)^{\otimes 2}}\otimes T^\beta_{\chi_{22}})( \bigtriangleup_\beta\otimes\mathrm{id}_{T^c(V)^{\otimes 2}}) ,\\[5pt]
\lefteqn{(\bigtriangleup_\beta\otimes\mathrm{id}_{T^c(V)^{\otimes
2}})T^\beta_{\chi_{22}} }\\[3pt]
&=&(\mathrm{id}_{T^c(V)^{\otimes 2}}\otimes
T^\beta_{\chi_{22}})(T^\beta_{\chi_{22}}\otimes
\mathrm{id}_{T^c(V)^{\otimes 2}})(\mathrm{id}_{T^c(V)^{\otimes
2}}\otimes \bigtriangleup_\beta).
\end{array} \right.\]
On $V^{\otimes i}\underline{\otimes}V^{\otimes
j}\underline{\otimes}V^{\otimes k}\underline{\otimes}V^{\otimes
l}$, we have
$T^\beta_{\chi_{22}}=T^\sigma_{\chi_{i+j,k+l}}=\beta_{i+j,k+l}$.

So on $V^{\otimes i_1+i_2}\underline{\otimes}V^{\otimes
j_1+j_2}\underline{\otimes}V^{\otimes
r}\underline{\otimes}V^{\otimes s}$,
\begin{eqnarray*}
\lefteqn{(\mathrm{id}_V^{\otimes r+s}\otimes
\bigtriangleup_{\beta\ (i_1,j_1,
i_2,j_2)}\beta_{i_1+j_1+i_2+j_2,r+s})}\\[3pt]
&=&(\beta_{i_1+j_1,r+s}\otimes\mathrm{id}_V^{\otimes
i_2+j_2})(\mathrm{id}_V^{\otimes
i_1+j_1}\otimes\beta_{i_2+j_2,r+s})(\bigtriangleup_{\beta\
(i_1,j_1, i_2,j_2)}\otimes \mathrm{id}_V^{\otimes r+s}),
\end{eqnarray*}
and on $V^{\otimes i}\underline{\otimes}V^{\otimes
j}\underline{\otimes}V^{\otimes k}\underline{\otimes}V^{\otimes
l}$,
\begin{eqnarray*}
\lefteqn{(\bigtriangleup_{\beta\ (i_1,j_1, i_2,j_2)}\otimes
\mathrm{id}_V^{\otimes
r+s})\beta_{r+s,i_1+j_1+i_2+j_2}}\\[3pt]
&=&(\mathrm{id}_V^{\otimes
i_1+j_1}\otimes\beta_{r+s,i_2+j_2})(\beta_{r+s,i_1+j_1}\otimes\mathrm{id}_V^{\otimes
i_2+j_2})(\mathrm{id}_V^{\otimes r+s}\otimes
\bigtriangleup_{\beta\ (i_1,j_1, i_2,j_2)}).
\end{eqnarray*}
In order to prove our lemma, we use induction on $k$ and the above
formulas for $r=l$ and $s=0$.

The cases $k=1$ and $k=2$ are trivial.
\begin{eqnarray*}
\lefteqn{\beta_{i_1+j_1+\cdots+i_{k+1}+j_{k+1},l}(\bigtriangleup_{\beta\
(i_1,j_1,\ldots, i_k,j_k)}^{(k)}\otimes \mathrm{id}_V^{\otimes
l})}\\[3pt]
&=&(\beta_{i_1+i_2+j_1+j_2,l}\otimes \mathrm{id}_V^{\otimes
i_3+j_3+\cdots+j_{k+1}})(\mathrm{id}_V^{\otimes
i_1+j_1+i_2+j_2}\otimes \beta_{i_3+j_3+\cdots+j_{k+1},l})\\[3pt]
&&\circ(\bigtriangleup_{\beta\ (i_1,j_1, i_2,j_2)}\otimes
\mathrm{id}_V^{\otimes
i_3+j_3+\cdots+j_{k+1}+l})(\bigtriangleup_{\beta\
(i_1+i_2,j_1+j_2,i_3,j_3,\ldots,
i_{k+1},j_{k+1})}^{(k-1)}\otimes \mathrm{id}_V^{\otimes l})\\[3pt]
&=&(\beta_{i_1+i_2+j_1+j_2,l}\otimes \mathrm{id}_V^{\otimes
i_3+j_3+\cdots+j_{k+1}})(\bigtriangleup_{\beta\ (i_1,j_1,
i_2,j_2)}\otimes \mathrm{id}_V^{\otimes
i_3+j_3+\cdots+j_{k+1}+l})\\[3pt]
&&\circ(\mathrm{id}_V^{\otimes i_1+j_1+i_2+j_2}\otimes
\beta_{i_3+j_3+\cdots+j_{k+1},l})(\bigtriangleup_{\beta\
(i_1+i_2,j_1+j_2,i_3,j_3,\ldots,
i_{k+1},j_{k+1})}^{(k-1)}\otimes \mathrm{id}_V^{\otimes l})\\[3pt]
&=&(\beta_{i_1+j_1+i_2+j_2,l}(\bigtriangleup_{\beta\ (i_1,j_1,
i_2,j_2)}\otimes \mathrm{id}_V^{\otimes l})\otimes
\mathrm{id}_V^{\otimes
i_3+j_3+\cdots+j_{k+1}})\\[3pt]
&&\circ(\mathrm{id}_V^{\otimes i_1+j_1+i_2+j_2}\otimes
\beta_{i_3+j_3+\cdots+j_{k+1},l})(\bigtriangleup_{\beta\
(i_1+i_2,j_1+j_2,i_3,j_3,\ldots,
i_{k+1},j_{k+1})}^{(k-1)}\otimes \mathrm{id}_V^{\otimes l})\\[3pt]
&=&((\mathrm{id}_V^{\otimes l}\otimes \bigtriangleup_{\beta\
(i_1,j_1, i_2,j_2)})\beta_{i_1+j_1+i_2+j_2,l}\otimes
\mathrm{id}_V^{\otimes
i_3+j_3+\cdots+j_{k+1}})\\[3pt]
&&\circ(\mathrm{id}_V^{\otimes i_1+j_1+i_2+j_2}\otimes
\beta_{i_3+j_3+\cdots+j_{k+1},l})(\bigtriangleup_{\beta\
(i_1+i_2,j_1+j_2,i_3,j_3,\ldots,
i_{k+1},j_{k+1})}^{(k-1)}\otimes \mathrm{id}_V^{\otimes l})\\[3pt]
&=&(\mathrm{id}_V^{\otimes l}\otimes \bigtriangleup_{\beta\
(i_1,j_1, i_2,j_2)}\otimes \mathrm{id}_V^{\otimes
i_3+j_3+\cdots+j_{k+1}})\\[3pt]
&&\circ \beta_{i_1+j_1+\cdots+j_{k+1},l}(\bigtriangleup_{\beta\
(i_1+i_2,j_1+j_2,i_3,j_3,\ldots,
i_{k+1},j_{k+1})}^{(k-1)}\otimes \mathrm{id}_V^{\otimes l})\\[3pt]
&=&(\mathrm{id}_V^{\otimes l}\otimes \bigtriangleup_{\beta\
(i_1,j_1, i_2,j_2)}\otimes \mathrm{id}_V^{\otimes
i_3+j_3+\cdots+j_{k+1}})\\[3pt]
&&\circ ( \mathrm{id}_V^{\otimes l}\otimes \bigtriangleup_{\beta\
(i_1+i_2,j_1+j_2,i_3,j_3,\ldots,
i_{k+1},j_{k+1})}^{(k-1)})\beta_{i_1+j_1+\cdots+j_{k+1},l}\\[3pt]
&=&(\mathrm{id}_V^{\otimes l}\otimes \bigtriangleup_{\beta\
(i_1,j_1,\ldots,
i_{k+1},j_{k+1})}^{(k)})\beta_{i_1+j_1+\cdots+i_{k+1}+j_{k+1},l}.
\end{eqnarray*}
The another equality can be proved similarly.
\end{proof}

\begin{proposition}Let $(V,M,\sigma)$ be a quantum multi-brace algebra. Then we have \[\left\{
\begin{array}{lll}
\beta(\ast\otimes \mathrm{id}_{T^c(V)})&=&(
\mathrm{id}_{T^c(V)}\otimes \ast)\beta_1\beta_2 ,\\[3pt]
\beta(\mathrm{id}_{T^c(V)}\otimes \ast)&=&(\ast\otimes
\mathrm{id}_{T^c(V)} )\beta_2\beta_1,
\end{array} \right.
\]where $\ast=\varepsilon\otimes \varepsilon+\sum_{r\geq 1}M^{\otimes r}\circ \bigtriangleup_\beta^{(r-1)}$.\end{proposition}
\begin{proof}We only need to verify that for all $k,l\geq 1$,
\[\left\{
\begin{array}{lll}
\lefteqn{\beta_{kl}\big((M_{(i_1,j_1,\ldots, i_k,j_k)}\circ
\bigtriangleup_{\beta\ (i_1,j_1,\ldots, i_k,j_k)}^{(k-1)})\otimes
\mathrm{id}_V^{\otimes
l}\big)}\\[3pt]
&=&\big(\mathrm{id}_V^{\otimes l}\otimes (M_{(i_1,j_1,\ldots,
i_k,j_k)}\circ \bigtriangleup_{\beta\ (i_1,j_1,\ldots,
i_k,j_k)}^{(k-1)})\big)\beta_{i_1+j_1+\cdots+i_k+j_k,l} ,\\[3pt]
\lefteqn{\beta_{lk}\big(\mathrm{id}_V^{\otimes l}\otimes
(M_{(i_1,j_1,\ldots, i_k,j_k)}\circ \bigtriangleup_{\beta\
(i_1,j_1,\ldots,
i_k,j_k)}^{(k-1)})\big)}\\[3pt]
&=&\big(( M_{(i_1,j_1,\ldots, i_k,j_k)}\circ
\bigtriangleup_{\beta\ (i_1,j_1,\ldots,
i_k,j_k)}^{(k-1)})\otimes\mathrm{id}_V^{\otimes
l}\big)\beta_{l,i_1+j_1+\cdots+i_k+j_k}.
\end{array} \right.
\]They follow from the preceding lemmas immediately.\end{proof}

\begin{theorem}Let $(V,M,\sigma)$ be a quantum multi-brace algebra. Then $(T(V), \ast,\beta)$ is a braided algebra. \end{theorem}
\begin{proof}We only need to show that $\ast$ is associative. First we show that $\ast(\ast\otimes \mathrm{id}_{T^c(V)})$ and
$\ast(\mathrm{id}_{T^c(V)}\otimes \ast)$ are coalgebra maps from
$(T^c(V)^{\otimes 3},\bigtriangleup_{\beta,3})$ to $T^c(V)$.

We have
\begin{eqnarray*}
\lefteqn{\delta\circ\ast(\ast\otimes \mathrm{id}_{T^c(V)})}\\[3pt]
&=&(\ast\otimes \ast)\circ \bigtriangleup_\beta\circ(\ast\otimes \mathrm{id}_{T^c(V)})\\[3pt]
&=&(\ast\otimes \ast)\circ \beta_2\circ \delta^{\otimes 2}\circ(\ast\otimes \mathrm{id}_{T^c(V)})\\[3pt]
&=&(\ast\otimes \ast)\circ \beta_2\circ (\delta\ast\otimes \delta)\\[3pt]
&=&(\ast\otimes \ast)\circ \beta_2\circ ((\ast\otimes \ast)\circ \bigtriangleup_\beta\otimes \delta)\\[3pt]
&=&(\ast\otimes \ast)\circ \beta_2\circ (\ast\otimes \ast\otimes \mathrm{id}_{T^c(V)}\otimes \mathrm{id}_{T^c(V)})\circ \beta_2\circ\beta^{\otimes 3}\\[3pt]
&=&(\ast\otimes \ast)\circ (\ast\otimes \beta(\ast\otimes
\mathrm{id}_{T^c(V)})\otimes \mathrm{id}_{T^c(V)})\circ
\beta_2\circ\delta^{\otimes 3}\\[3pt]
&=&(\ast\otimes \ast)\circ (\ast\otimes (
\mathrm{id}_{T^c(V)}\otimes \ast)\beta_1\beta_2\otimes
\mathrm{id}_{T^c(V)})\circ \beta_2\circ\delta^{\otimes 3}\\[3pt]
&=&(\ast\otimes \ast)\circ (\ast\otimes
\mathrm{id}_{T^c(V)}\otimes \ast\otimes \mathrm{id}_{T^c(V)})\circ
\beta_3\beta_4\beta_2\circ\delta^{\otimes 3}\\[3pt]
&=&(\ast\otimes \ast)\circ (\ast\otimes
\mathrm{id}_{T^c(V)}\otimes \ast\otimes \mathrm{id}_{T^c(V)})\circ
T^\beta_{w_3^{-1}}\circ\delta^{\otimes 3}\\[3pt]
&=&(\ast(\ast\otimes \mathrm{id}_{T^c(V)})\otimes \ast(\ast\otimes
\mathrm{id}_{T^c(V)}))\bigtriangleup_{\beta,3}.
\end{eqnarray*}
The first and third equalities follow from the fact that $\ast:
T^c(V)\underline{\otimes}T^c(V)\rightarrow T^c(V)$ is a coalgebra
map.

Similarly, we can prove that $\ast(\mathrm{id}_{T^c(V)}\otimes
\ast)$ is also a coalgebra map.

Now we show that $P_V\circ \ast(\ast\otimes
\mathrm{id}_{T^c(V)})=P_V\circ \ast(\mathrm{id}_{T^c(V)}\otimes
\ast)$.

On $V^{\otimes i}\underline{\otimes}V^{\otimes
j}\underline{\otimes}V^{\otimes k}$, we have
\begin{eqnarray*}
\lefteqn{P_V\circ\big( \ast(\ast\otimes
\mathrm{id}_{T^c(V)})\big)}\\
&=&P_V\Big(\sum_{s=1}^{i+j+k}M^{\otimes
s}\circ\bigtriangleup_\beta^{(s-1)}\circ\big(\sum_{r=1}^{i+j}(M^{\otimes
r}\circ \bigtriangleup_\beta^{(r-1)})\otimes
\mathrm{id}_V^{\otimes k}\big)\Big)\\
&=&\sum_{r=1}^{i+j}M_{rk}\circ \big((M^{\otimes r}\circ
\bigtriangleup_\beta^{(r-1)})\otimes
\mathrm{id}_V^{\otimes k}\big)\\
&=&\sum_{l=1}^{j+k}M_{il}\circ \big(\mathrm{id}_V^{\otimes
i}\otimes
(M^{\otimes l}\circ \bigtriangleup_\beta^{(l-1)})\big)\\
&=&P_V\Big(\sum_{s=1}^{i+j+k}M^{\otimes
s}\circ\bigtriangleup_\beta^{(s-1)}\circ\big(\sum_{l=1}^{j+k}\mathrm{id}_V^{\otimes
i}\otimes
(M^{\otimes l}\circ \bigtriangleup_\beta^{(l-1)})\big)\Big)\\
&=&P_V\circ \ast(\mathrm{id}_{T^c(V)}\otimes \ast),
\end{eqnarray*}where the third equality follows from the
associativity condition.

Finally, it is clear that both of $P_V\circ \ast(\ast\otimes
\mathrm{id}_{T^c(V)})$ and $P_V\circ
\ast(\mathrm{id}_{T^c(V)}\otimes \ast)$ vanish on
$1\underline{\otimes}1\underline{\otimes}1$. Then by the Corollary
4.3, we have that $\ast(\ast\otimes
\mathrm{id}_{T^c(V)})=\ast(\mathrm{id}_{T^c(V)}\otimes \ast)$. The
compatibility conditions for the unit and braiding are trivial.
\end{proof}

\begin{remark}By using the dual construction stated in Remark 2.3.3, we can easily define coalgebra structures on the tensor space $T(V)$ which provide braided coalgebras.\end{remark}

\begin{example}[Reconstruction of quantum shuffle algebras] Let
$(V,\sigma)$ be a braided vector space. Then $(V,M,\sigma)$ is a
multi-brace algebra with $M_{10}=\mathrm{id}_V=M_{01}$ and
$M_{pq}=0$ for other cases. The resulting algebra $T(V)$ in the
above theorem is just the quantum shuffle algebra, i.e.,
$\ast=\textrm{\cyr sh}_\sigma$.
\end{example}
\begin{example}[Quantum quasi-shuffle algebras] Let
$(V,m,\sigma)$ be a braided algebra. Then $(V,M,\sigma)$ is a
multi-brace algebra with $M_{10}=\mathrm{id}_V=M_{01}$, $M_{11}=m$
and $M_{pq}=0$ for other cases. The resulting algebra $T(V)$ in
the above theorem is called the \emph{quantum quasi-shuffle
algebra}. We denote by $\Join_\sigma$ the quantum quasi-shuffle
product. This new product has the following inductive relation:
for any $u_1,\ldots, u_i,v_1,\ldots, v_j\in V$,
\begin{eqnarray*}
\lefteqn{(u_1\otimes\cdots\otimes u_i)\Join_\sigma (v_1\otimes\cdots\otimes v_j)}\\
&=&u_1\otimes \Big((u_2\otimes\cdots\otimes u_i)\Join_\sigma (v_1\otimes\cdots\otimes v_{j})\Big)\\
&&+(\mathrm{id}_V\otimes \Join_{\sigma (i-1,j)})(\beta_{i,1}\otimes \mathrm{id}_V^{\otimes j-1})(u_1\otimes\cdots\otimes u_i\otimes v_1\otimes\cdots\otimes v_j)\nonumber\\
&&+(\mu\otimes\Join_{\sigma (i-1,j-1)} )(\mathrm{id}_V\otimes
\beta_{i-1,1}\otimes \mathrm{id}_V^{\otimes
j-1})(u_1\otimes\cdots\otimes u_i\otimes v_1\otimes\cdots\otimes
v_j),
\end{eqnarray*}
where  $\Join_{\sigma (i,j)}$ the restriction of $\Join_\sigma$ on
$V^{\otimes i}\underline{\otimes}V^{\otimes j}$. It is the
generalization of quantum shuffle algebra and the quantization of
the classical quasi-shuffle algebra. It is not hard to see that
Hoffman's q-deformation of quasi-shuffle product (see \cite{Hof2})
is a special quantum quasi-shuffle product.

\end{example}
\begin{proposition}Let $V$ be a Yetter-Drinfel'd module over a Hopf algebra $H$. If $V$ is both a module-algebra and comodule-algebra with multiplication $m_V$, then the quantum quasi-shuffle algebra built on $V$ is a module-algebra
with the diagonal action and a comodule-algebra with the diagonal
coaction.\end{proposition}
\begin{proof}We use induction to prove the statement. On $V\underline{\otimes}V$, $\Join_\sigma=m_V+\textrm{\cyr sh}_\sigma$. Since $T_\sigma(V)$ is both a module-algebra and a comodule-algebra with the diagonal action and coaction respectively, and $m_V$ is both a module map and comodule map, we get the result. By using the above inductive formula of quantum quasi-shuffles to reduce the degree, the rest of the proof follows from that $m_V$ is both a module map and comodule map.\end{proof}

\begin{remark}Under the assumptions in the above proposition, we can define a map $\Join_\sigma: T(V)\otimes T(V)\rightarrow T(V)$ by using the inductive formula. It is not difficult to prove by induction that this $\Join_\sigma$ defines an associative product on $T(V)$. By noticing that the natural braiding of the Yetter-Drinfel'd module $T(V)$ is just $\beta$, $T(V)$ satisfies all conditions of Proposition 2.4. Hence we can reprove directly that $(T(V),\Join_\sigma, \beta)$ is a braided algebra in this special case. \end{remark}

For more properties about the quantum quasi-shuffle algebra, one
can see \cite{JRZ}.

Let $(V,M,\sigma)$ be a quantum multi-brace algebra and $\ast$ be
the product constructed by $M$ and $\sigma$ as before. We denote
by $Q_\sigma(V)$ the subalgebra of $(T(V),\ast)$ generated by $V$.
If we define $\ast^n: V^{\underline{\otimes} n+1}\rightarrow T(V)$
by $v_1\underline{\otimes} \cdots\underline{\otimes}
v_{n+1}\mapsto v_1\ast \cdots\ast v_{n+1}$, and
$\ast^0=\mathrm{id}_V$ for convenience, then
$Q_\sigma(V)=K\oplus\oplus_{n\geq 0}\mathrm{Im}\ast^n$. This
algebra is a generalization of the quantum symmetric algebra over
$V$.

\begin{proposition}The pair $(Q_\sigma(V), \beta)$ is a braided algebra.\end{proposition}
\begin{proof}In order to prove the statement, we only need to verify that $\beta$ is a braiding on $Q_\sigma(V)$. In fact, we have that $\beta(\ast^k\otimes \ast^l)=(\ast^l\otimes \ast^k)\beta_{k+1,l+1}$. We use induction on $k+l$.

The case $k=l=0$ is trivial since $\sigma(\mathrm{id}_V\otimes
\mathrm{id}_V)=(\mathrm{id}_V\otimes \mathrm{id}_V)\sigma$.

When $k+l\geq 1$,
\begin{eqnarray*}\beta(\ast^k\otimes \ast^l)&=&\beta(\ast\otimes\mathrm{id}_{T(V)})(\mathrm{id}_V\otimes \ast^{k-1}\otimes
\ast^l)\\[3pt]
&=&(\mathrm{id}_{T(V)}\otimes\ast)\beta_1\beta_2(\mathrm{id}_V\otimes \ast^{k-1}\otimes \ast^l)\\[3pt]
&=&(\mathrm{id}_{T(V)}\otimes\ast)\beta_1\big(\mathrm{id}_V\otimes \beta(\ast^{k-1}\otimes \ast^l)\big)\\[3pt]
&=&(\mathrm{id}_{T(V)}\otimes\ast)\beta_1(\mathrm{id}_V\otimes \ast^l\otimes \ast^{k-1})(\mathrm{id}_V\otimes \beta_{k,l+1})\\[3pt]
&=&(\mathrm{id}_{T(V)}\otimes\ast)(\beta(\mathrm{id}_V\otimes \ast^l)\otimes \ast^{k-1})(\mathrm{id}_V\otimes \beta_{k,l+1})\\[3pt]
&=&(\ast^l\otimes \ast^k)(\beta_{1,l+1}\otimes\mathrm{id}_V^{\otimes k})(\mathrm{id}_V\otimes \beta_{k,l+1})\\[3pt]
&=&(\ast^l\otimes \ast^k)\beta_{k+1,l+1}.
\end{eqnarray*}
\end{proof}

For any quantum multi-brace algebra $(V,M,\sigma)$, if we endow
$T(V)$ with the usual grading, then the algebra $(T(V),\ast)$ is
not graded in general. But with this grading, we have:

\begin{proposition}The term of highest degree in the product $\ast$ is the quantum shuffle product. \end{proposition}
\begin{proof}We need to verify that for any $i,j\geq 1$, $$M^{\otimes i+j}\circ\bigtriangleup_\beta^{(i+j-1)}(u_1\otimes\cdots \otimes u_i\underline{\otimes}v_1\otimes\cdots \otimes v_j)=\sum_{w\in \mathfrak{S}_{ij}}T^\sigma_w(u_1\otimes\cdots \otimes u_i\underline{\otimes}v_1\otimes\cdots \otimes v_j).$$ We use induction on $i+j$.
When $i=j=1$, $M^{\otimes 2}\circ\bigtriangleup_\beta(u\otimes
v)=u\otimes v+\sigma(u\otimes v)=u\textrm{\cyr sh}_\sigma v$. By
inductive hypothesis, we have
\begin{eqnarray*}
\lefteqn{M^{\otimes i+j}\circ\bigtriangleup_\beta^{(i+j-1)}(u_1\otimes\cdots \otimes u_i\underline{\otimes}v_1\otimes\cdots \otimes v_j)}\\[3pt]
&=&\Big((M^{\otimes i+j-1}\circ\bigtriangleup_\beta^{(i+j-2)})\otimes M\Big)\bigtriangleup_\beta(u_1\otimes\cdots \otimes u_i\underline{\otimes}v_1\otimes\cdots \otimes v_j)\\[3pt]
&=&\Big((M^{\otimes i+j-1}\circ\bigtriangleup_\beta^{(i+j-2)})\otimes M\Big)(u_1\otimes\cdots \otimes u_i\underline{\otimes}v_1\otimes\cdots \otimes v_{j-1}\underline{\otimes}1\underline{\otimes}v_j\\[3pt]
&&+u_1\otimes\cdots \otimes u_{i-1}\underline{\otimes}\beta_{1j}(u_i\underline{\otimes}v_1\otimes\cdots \otimes v_j)\underline{\otimes}1)\\[3pt]
&=&\big(\sum_{w\in \mathfrak{S}_{i,j-1}}T^\sigma_w\otimes \mathrm{id}_V+\sum_{w^\prime\in \mathfrak{S}_{i-1,j}}(T^\sigma_w\otimes \mathrm{id}_V)\sigma_{i+j-1}\cdots\sigma_i\big)(u_1\otimes \cdots\otimes v_j)\\[3pt]
&=&(u_1\otimes\cdots \otimes u_i)\textrm{\cyr
sh}_\sigma(v_1\otimes\cdots \otimes v_j),
\end{eqnarray*}where the second equality follows from the fact that $M^{\otimes k}\bigtriangleup_\beta^{(k-1)}(x)=0$ for any $x$ with degree smaller than $k$, and the fourth equality follows from the fact that for any $w\in \mathfrak{S}_{i,j}$ we have either $w(i+j)=i+j$ or $w(i)=i+j$.\end{proof}

From the classical theory (see, e.g., \cite{LoR}), we also know
that $(T^c(V),\ast)$ has an antipode $S$ given by $S(1)=1$ and
$S(x)=\sum_{n\geq 0}(-1)^{n+1}\ast^{\otimes n}\circ
\overline{\delta}^{(n)}(x)$ for any $x\in
\mathrm{Ker}\varepsilon$.

\section{Constructions of quantum multi-brace algebras}
Since the conditions in the definition of quantum multi-brace
algebras are a little bit complicated, it seems that it is not
easy to obtain the map $M$. We now introduce a new notion
motivated by \cite{LoR} and use it to provide quantum multi-brace
algebras.
\begin{definition}A \emph{unital 2-braided algebra} is a braided vector space $(V,\sigma)$ equipped with two associative algebra structures $\ast$ and $\cdot$, which share the same unit, such that both $(V,\ast,\sigma)$ and $(V,\cdot,\sigma)$ are braided algebras. We denote a 2-braided algebra by $(V, \ast,\cdot,\sigma)$.\end{definition}

\begin{example}1. Let $(A, m,\alpha)$ be a braided algebra. Then $(A,
m,m, \alpha)$ is a trivial unital 2-braided algebra.

2. Let $(V,\sigma)$ be a braided vector space. Then $(T(V), m,
\textrm{\cyr sh}_\sigma, \beta)$ is a unital 2-braided algebra,
where $m$ is the concatenation product.

\end{example}

Let $(V, \ast,\cdot,\sigma)$ be a unital 2-braided algebra. We
denote by $\cdot^k$ the map from $V^{\otimes k+1}$ to $V$ given by
$v_1\otimes \cdots \otimes v_{k+1}\mapsto v_1 \cdot \cdots \cdot
v_{k+1}$. We define $M_{pq}:V^{\otimes p}\otimes V^{\otimes
q}\rightarrow V$ for $ p, q\geq 0$ inductively as follows:
\[\left\{
\begin{array}{lllll}
M_{00}&=&0,&& \\
M_{10}&=&\mathrm{id}_V&=&M_{01},\\
M_{n0}&=&0&=&M_{0n},\ \mathrm{for}\ n\geq 2,
\end{array} \right.
\]
and

\begin{eqnarray*}\lefteqn{M_{pq}(u_1\otimes \cdots \otimes u_p\underline{\otimes }v_1 \otimes \cdots \otimes v_q)}\\[3pt]
&=&(u_1\cdot \cdots \cdot u_p)\ast(v_1\cdot \cdots \cdot
v_q)\\[3pt]
&&-\sum_{k=2}^{p+q}\sum_{I_k, J_k}\cdot^{k-1}M_{(i_i,j_1,\ldots,
i_k,j_k)}\circ \bigtriangleup_{\beta\ (i_1,j_1,\ldots,
i_{k},j_{k})}^{(k-1)}(u_1\otimes \cdots \otimes
u_p\underline{\otimes }v_1 \otimes \cdots \otimes
v_q),\end{eqnarray*} where $I_k=(i_1,\ldots, i_k)$ and
$J_k=(j_1,\ldots, j_k)$ run through all the partitions of length
$k$ of $p$ and $q$ respectively.

For instance, \begin{eqnarray*}
M_{11}(u\underline{\otimes} v)&=&u\ast v\\[3pt]
&&-\cdot(M_{01}\otimes M_{10})(1\otimes \sigma(u\otimes v)\otimes 1)\\[3pt]
&&-\cdot(M_{10}\otimes M_{01})(u\otimes \sigma(1\otimes 1)\otimes v)\\[3pt]
&=&u\ast v-\cdot\sigma(u\otimes v)-u\cdot v,
\end{eqnarray*}
\begin{eqnarray*}
M_{21}(u \otimes v\underline{\otimes}w)&=&(u\cdot v)\ast w\\[3pt]
&&-u\cdot M_{11}(v\underline{\otimes} w)-\cdot(M_{11}\otimes \mathrm{id}_V)(u\otimes \sigma(v\otimes w))\\[3pt]
&&-\cdot^{2}\big(u\otimes v\otimes w+\sigma_2(u\otimes
v\otimes w)+\sigma_1\sigma_2(u\otimes v\otimes w)\big)\\[3pt]
&=&(u\cdot v)\ast w-u\cdot(v\ast w)\\[3pt]
&&+\cdot^{2}\sigma_2(u\otimes v\otimes
w)-\cdot(\ast\otimes\mathrm{id}_V)\sigma_2(u\otimes v\otimes w),
\end{eqnarray*}
and
\begin{eqnarray*}
M_{12}(u \underline{\otimes}v\otimes w)&=&u\ast( v\cdot w)-(u\ast v)\cdot w\\[3pt]
&&+\cdot^{2}\sigma_1(u\otimes v\otimes
w)-\cdot(\mathrm{id}_V\otimes\ast)\sigma_1(u\otimes v\otimes w).
\end{eqnarray*}

\begin{theorem}Let $(V, \ast,\cdot,\sigma)$ be a unital 2-braided algebra and $M=(M_{pq})$ be the maps defined above. Then $(V, M, \sigma)$ is a quantum multi-brace algebra.\end{theorem}
\begin{proof}First we verify the Yang-Baxter condition. We use induction on $i+j+k$.

When $i=j=k=1$,
\begin{eqnarray*}
\beta_{11}(M_{11}\otimes \mathrm{id}_V)&=&\sigma(\ast\otimes\mathrm{id}_V-(\cdot\otimes\mathrm{id}_V)\sigma_1-\cdot\otimes\mathrm{id}_V
)\\[3pt]
&=&(\mathrm{id}_V\otimes\ast)\sigma_1\sigma_2-(\mathrm{id}_V\otimes\cdot)\sigma_1\sigma_2\sigma_1-(\mathrm{id}_V\otimes\cdot)\sigma_1\sigma_2\\[3pt]
&=&(\mathrm{id}_V\otimes\ast)\sigma_1\sigma_2-(\mathrm{id}_V\otimes\cdot)\sigma_2\sigma_1\sigma_2-(\mathrm{id}_V\otimes\cdot)\sigma_1\sigma_2\\[3pt]
&=&\big(\mathrm{id}_V\otimes(\ast-\cdot\sigma-\cdot)\big)\sigma_1\sigma_2\\[3pt]
&=&( \mathrm{id}_V\otimes M_{11})\beta_{21}.
\end{eqnarray*}For general case, we have
\begin{eqnarray*}
\lefteqn{\beta_{1k}(M_{pq}\otimes \mathrm{id}_V^{\otimes
k})}\\[3pt]
&=&\beta_{1k}(\ast\otimes \mathrm{id}_V^{\otimes k})(\cdot^{
p-1}\otimes \cdot^{ q-1}\otimes \mathrm{id}_V^{\otimes
k})\\[3pt]
&&-\sum \beta_{1k} (\cdot^{ r-1}\otimes \mathrm{id}^{\otimes
l})\Big((M_{(i_i,j_1,\ldots, i_r,j_r)}\circ\bigtriangleup_{\beta\
(i_1,j_1,\ldots,
i_{r},j_{r})})\otimes \mathrm{id}_V^{\otimes l}\Big)\\[3pt]
&=&( \mathrm{id}_V^{\otimes k}\otimes\ast)(\beta_{1k}\otimes
\mathrm{id}_V )(\mathrm{id}_V\otimes \beta_{1k})(\cdot^{
p-1}\otimes \cdot^{ q-1}\otimes \mathrm{id}_V^{\otimes
k})\\[3pt]
&&-\sum (\mathrm{id}^{\otimes l}\otimes \cdot^{
r-1})\beta_{rk}\Big((M_{(i_i,j_1,\ldots,
i_r,j_r)}\circ\bigtriangleup_{\beta\ (i_1,j_1,\ldots,
i_{r},j_{r})})\otimes \mathrm{id}^{\otimes l}\Big)\\[3pt]
&=&( \mathrm{id}_V^{\otimes k}\otimes\ast)( \mathrm{id}_V^{\otimes
k}\otimes\cdot^{
p-1}\otimes \cdot^{ q-1})\beta_{p+q,k}\\[3pt]
&&-\sum (\mathrm{id}^{\otimes l}\otimes \cdot^{
r-1})\Big((M_{(i_i,j_1,\ldots, i_r,j_r)}\circ
\bigtriangleup_{\beta\ (i_1,j_1,\ldots,
i_{r},j_{r})})\otimes \mathrm{id}_V^{\otimes l}\Big)\beta_{p+q,k}\\[3pt]
&=&( \mathrm{id}_V^{\otimes k}\otimes M_{pq})\beta_{p+q,k}.
\end{eqnarray*}

The condition $\beta_{i1}(\mathrm{id}_V^{\otimes i}\otimes
M_{jk})=(M_{jk}\otimes \mathrm{id}_V^{\otimes i} )\beta_{i,j+k}$
can be verified similarly.

Now we want to prove that $M=(M_{pq})$ also satisfies the
associativity condition. We use induction on $i+j+k$.

When $i=j=k=1$, the associativity condition is just
$M_{11}(M_{11}\otimes
\mathrm{id}_V)+M_{21}+M_{21}\sigma_1=M_{11}(\mathrm{id}_V \otimes
M_{11} )+M_{12}+M_{12}\sigma_2$. Now we verify it:
\begin{eqnarray*}
\lefteqn{M_{11}(M_{11}\otimes \mathrm{id}_V)+M_{21}+M_{21}\sigma_1}\\[3pt]
&=&\ast^{ 2}-\cdot\sigma(\ast\otimes \mathrm{id}_V)-\cdot(\ast\otimes \mathrm{id}_V)\\[3pt]
&&-\ast(\cdot\otimes \mathrm{id}_V)\sigma_1+\cdot\sigma(\cdot\otimes \mathrm{id}_V)\sigma_1+\cdot^{ 2}\sigma_1\\[3pt]
&&-\ast(\cdot\otimes \mathrm{id}_V)+\cdot\sigma(\cdot\otimes \mathrm{id}_V)+\cdot^{ 2}\\[3pt]
&&+\ast(\cdot\otimes \mathrm{id}_V)-\cdot (\mathrm{id}_V\otimes \ast)+\cdot^{ 2}\sigma_2-\cdot(\ast\otimes \mathrm{id}_V)\sigma_2\\[3pt]
&&+\ast(\cdot\otimes \mathrm{id}_V)\sigma_1-\cdot
(\mathrm{id}_V\otimes \ast)\sigma_1+\cdot^{
2}\sigma_2\sigma_1-\cdot(\ast\otimes
\mathrm{id}_V)\sigma_2\sigma_1\\[3pt]
&=&\ast^{ 2}-\cdot(\mathrm{id}_V\otimes \ast)\sigma_1\sigma_2-\cdot(\ast\otimes \mathrm{id}_V)\\[3pt]
&&+\cdot^{2}\sigma_1\sigma_2\sigma_1+\cdot^{ 2}\sigma_1+\cdot^{ 2}\sigma_1\sigma_2+\cdot^{ 2}\\[3pt]
&&-\cdot (\mathrm{id}_V\otimes \ast)+\cdot^{ 2}\sigma_2-\cdot(\ast\otimes \mathrm{id}_V)\sigma_2\\[3pt]
&&-\cdot (\mathrm{id}_V\otimes \ast)\sigma_1+\cdot^{
2}\sigma_2\sigma_1-\cdot\sigma(
\mathrm{id}_V\otimes\ast)\\[3pt]
&=&\ast^{ 2}-\cdot\sigma(\mathrm{id}_V\otimes \ast)-\cdot (\mathrm{id}_V\otimes \ast)\\[3pt]
&&-\ast(\mathrm{id}_V\otimes \cdot)\sigma_2+\cdot\sigma(\mathrm{id}_V\otimes \cdot)\sigma_2+\cdot ^{ 2}\sigma_2\\[3pt]
&&-\ast(\mathrm{id}_V\otimes \cdot)+\cdot\sigma(\mathrm{id}_V\otimes \cdot)+\cdot (\mathrm{id}_V\otimes \cdot)\\[3pt]
&&+\ast(\mathrm{id}_V\otimes\cdot)-\cdot(\ast\otimes\mathrm{id}_V)+\cdot^{
2}\sigma_1-\cdot(\mathrm{id}_V\otimes\ast)\sigma_1\\[3pt]
&&+\ast(\mathrm{id}_V\otimes\cdot)\sigma_2-\cdot(\ast\otimes\mathrm{id}_V)\sigma_2+\cdot^{
2}\sigma_1\sigma_2-\cdot(\mathrm{id}_V\otimes\ast)\sigma_1\sigma_2\\[3pt]
&=&M_{11}(\mathrm{id}_V \otimes M_{11} )+M_{12}+M_{12}\sigma_2.
\end{eqnarray*}
For $i+j+k\geq 2$, we have
\begin{eqnarray*}
\lefteqn{\sum_{r=1}^{i+j}M_{rk}\circ \big((M^{\otimes r}\circ
\bigtriangleup_\beta^{(r-1)})\otimes
\mathrm{id}_V^{\otimes k}\big)}\\[3pt]
&=&\sum_{r\geq 1}\big(\ast(\cdot^{r-1}\otimes
\cdot^{r-1})-\sum_{l\geq 2}\cdot^{l-1}M^{\otimes l}\circ
\bigtriangleup_\beta^{(l-1)})\big)\circ \big((M^{\otimes r}\circ
\bigtriangleup_\beta^{(r-1)})\otimes
\mathrm{id}_V^{\otimes k}\big)\\[3pt]
&=&\ast\big((\sum_{r\geq 1}\cdot^{r-1}M^{\otimes r}\circ
\bigtriangleup_\beta^{(r-1)})\otimes \cdot^{k-1}\big)\\[3pt]
&&-\sum_{r\geq 1}\sum_{l\geq 2}\cdot^{l-1}M^{\otimes l}\circ
\bigtriangleup_\beta^{(l-1)}\circ \big((M^{\otimes r}\circ
\bigtriangleup_\beta^{(r-1)})\otimes
\mathrm{id}_V^{\otimes k}\big)\\[3pt]
&=&\ast\big(\ast(\cdot^{i-1}\otimes\cdot^{j-1})\otimes \cdot^{k-1}\big)\\[3pt]
&&-\sum_{r\geq 1}\sum_{l\geq 2}\cdot\big((\cdot^{l-2}M^{\otimes
l-1}\circ \bigtriangleup_\beta^{(l-2)})\otimes M\big)
\bigtriangleup_\beta\circ \big((M^{\otimes r}\circ
\bigtriangleup_\beta^{(r-1)})\otimes
\mathrm{id}_V^{\otimes k}\big)\\[3pt]
&=&\ast(\ast\otimes \mathrm{id}_V)(\cdot^{i-1}\otimes\cdot^{j-1}\otimes \cdot^{k-1})\\[3pt]
&&-\sum_{r\geq
1}\sum_{\underline{2}}\cdot\big(\ast(\cdot^{p_1-1}\otimes\cdot^{q_1-1})\otimes
M_{p_2,q_2}\big) \\[3pt]
&&\quad\quad\quad\quad\quad \circ \bigtriangleup_{\beta\
(p_1,q_1,p_2,q_2)}\circ\big((M^{\otimes r}\circ
\bigtriangleup_\beta^{(r-1)})\otimes
\mathrm{id}_V^{\otimes k}\big)\\[3pt]
&=&\ast(\ast\otimes \mathrm{id}_V)(\cdot^{i-1}\otimes\cdot^{j-1}\otimes \cdot^{k-1})\\[3pt]
&&-\sum_{r\geq 1}\sum_{\underline{2}}\cdot\big(\ast(\cdot^{p_1-1}\otimes\cdot^{q_1-1})\otimes M_{p_2,q_2}\big)\circ(\mathrm{id}_V^{\otimes p_1}\otimes \beta_{p_2,q_1}\otimes\mathrm{id}_V^{\otimes q_2})\\[3pt]
&&\quad\quad\quad\circ(\sum M_{(r_1,s_1,\ldots,
r_{p_1},s_{p_1})}\bigtriangleup_{\beta\ (r_1,s_1,\ldots,
r_{p_1},s_{q_1})}^{(p_1-1)}\\[3pt]
&&\quad\quad\quad\quad\quad\otimes M_{(r_{p_1+1},s_{p_1+1},\ldots,
r_{p_1+p_2},s_{p_1+p_2})}\bigtriangleup_{\beta\
(r_{p_1+1},s_{p_1+1},\ldots,
r_{p_1+p_2},s_{p_1+p_2})}^{(p_2-1)}\\[3pt]
&&\quad\quad\quad\quad\quad\otimes\mathrm{id}_V^{\otimes q_1}\otimes\mathrm{id}_V^{\otimes q_2})\\[3pt]
&&\quad\quad\quad\circ(\bigtriangleup_{\beta\
(r_1+\cdots+r_{p_1},s_1+\cdots+s_{p_1},
r_{p_1+1}+\cdots+r_{p_1+p_2},s_{p_1+1}+\cdots+s_{p_1+p_2})}\otimes \mathrm{id}_V^{\otimes k})\\[3pt]
&=&\ast(\ast\otimes \mathrm{id}_V)(\cdot^{i-1}\otimes\cdot^{j-1}\otimes \cdot^{k-1})\\
&&-\sum_{r\geq 1}\sum_{\underline{2}}\cdot\big(\ast(\cdot^{p_1-1}\otimes\cdot^{q_1-1})\otimes M_{p_2,q_2}\big) \\[3pt]
&&\quad\quad\circ(\sum M_{(r_1,s_1,\ldots,
r_{p_1},s_{p_1})}\bigtriangleup_{\beta \ (r_1,s_1,\ldots,
r_{p_1},s_{q_1})}^{(p_1-1)}\otimes\mathrm{id}_V^{\otimes q_1}\\[3pt]
&&\quad\quad\quad\otimes M_{(r_{p_1+1},s_{p_1+1},\ldots,
r_{p_1+p_2},s_{p_1+p_2})}\bigtriangleup_{\beta\
(r_{p_1+1},s_{p_1+1},\ldots,
r_{p_1+p_2},s_{p_1+p_2})}^{(p_2-1)}\otimes\mathrm{id}_V^{\otimes q_2})\\[3pt]
&&\quad\quad\circ(\mathrm{id}_V^{\otimes r_1+\cdots+s_{p_1}}\otimes \beta_{r_{p_1+1}+\cdots+s_{p_1+p_2},q_1}\otimes\mathrm{id}_V^{\otimes q_2})\\[3pt]
&&\quad\quad\circ(\bigtriangleup_{\beta\
(r_1+\cdots+r_{p_1},s_1+\cdots+s_{p_1},
r_{p_1+1}+\cdots+r_{p_1+p_2},s_{p_1+1}+\cdots+s_{p_1+p_2})}\otimes \mathrm{id}_V^{\otimes k})\\[3pt]
&=&\ast(\ast\otimes \mathrm{id}_V)(\cdot^{i-1}\otimes\cdot^{j-1}\otimes \cdot^{k-1})\\[3pt]
&&-\sum_{r\geq 1}\sum_{\underline{2}}\cdot\big(\ast(\cdot^{p_1-1}\otimes\cdot^{q_1-1})\otimes M_{p_2,q_2}\big) \\[3pt]
&&\quad\quad\circ(\sum M_{(r_1,s_1,\ldots,
r_{p_1},s_{p_1})}\bigtriangleup_{\beta\ (r_1,s_1,\ldots,
r_{p_1},s_{q_1})}^{(p_1-1)}\otimes\mathrm{id}_V^{\otimes q_1}\\[3pt]
&&\quad\quad\quad\quad\otimes M_{(r_{p_1+1},s_{p_1+1},\ldots,
r_{p_1+p_2},s_{p_1+p_2})}\bigtriangleup_{\beta \
(r_{p_1+1},s_{p_1+1},\ldots,
r_{p_1+p_2},s_{p_1+p_2})}^{(p_2-1)}\otimes\mathrm{id}_V^{\otimes q_2})\\[3pt]
&&\quad\quad\circ \bigtriangleup_{\beta,3,
(r_1+\cdots+r_{p_1},s_1+\cdots+s_{p_1},q_1,
r_{p_1+1}+\cdots+r_{p_1+p_2},s_{p_1+1}+\cdots+s_{p_1+p_2},q_2)}\\[3pt]
&=&\ast(\ast\otimes \mathrm{id}_V)(\cdot^{i-1}\otimes\cdot^{j-1}\otimes \cdot^{k-1})\\[3pt]
&&-\cdot\sum_{p+q+r< i+j+k}\big(\ast(\ast\otimes \mathrm{id}_V)(\cdot^{i-p-1}\otimes\cdot^{j-q-1}\otimes \cdot^{k-r-1})\\[3pt]
&&\quad\quad\quad\quad\otimes \sum_{s\geq 1}M_{sr}\circ
\big((M^{\otimes s}\circ \bigtriangleup_\beta^{(s-1)})\otimes
\mathrm{id}_V^{\otimes s}\big)\circ \bigtriangleup_{\beta,3, (i-p,j-q,k-r, p,q,r)}\\[3pt]
&=&\ast(\mathrm{id}_V\otimes \ast)(\cdot^{i-1}\otimes\cdot^{j-1}\otimes \cdot^{k-1})\\[3pt]
&&-\cdot\sum_{p+q+r< i+j+k}\big(\ast(\mathrm{id}_V\otimes \ast)(\cdot^{i-p-1}\otimes\cdot^{j-q-1}\otimes \cdot^{k-r-1})\\[3pt]
&&\quad\quad\quad\quad\otimes \sum_{s\geq 1}M_{ps}\circ
\big(\mathrm{id}_V^{\otimes p}\otimes( M^{\otimes s}\circ
\bigtriangleup_\beta^{(s-1)})\big)\circ \bigtriangleup_{\beta,3, (i-p,j-q,k-r, p,q,r)}\\[3pt]
&=&\sum_{l=1}^{j+k}M_{il}\circ\big(\mathrm{id}_V^{\otimes
i}\otimes (M^{\otimes l}\circ \bigtriangleup_\beta^{(l-1)})\big),
\end{eqnarray*}
where the third equality follows from the inductive hypothesis and
the associativity of $\ast$. Here
$\bigtriangleup_{\beta,3}=T^\beta_{w_3^{-1}}\circ \delta^{\otimes
3}$, and $\bigtriangleup_{\beta,3, (i,j,k,l,m,n)}$ is denoted by
the composition of $\bigtriangleup_{\beta,3}: V^{\otimes
i+k}\underline{\otimes}V^{\otimes
j+m}\underline{\otimes}V^{\otimes l+n}\rightarrow
T(V)^{\underline{\otimes} 6}$ with the projection from
$T(V)^{\underline{\otimes} 6}$ to $V^{\otimes
i}\underline{\otimes}V^{\otimes j}\underline{\otimes}V^{\otimes
k}\underline{\otimes}V^{\otimes l}\underline{\otimes}V^{\otimes
m}\underline{\otimes}V^{\otimes n}$.
\end{proof}

Let $A_{2-\mathrm{braided}}$ be the category of unital 2-braided
algebras and $A_{\mathcal{QMB}}$ be the category of quantum
multi-brace algebras. By the above proposition, we get a
functor\begin{eqnarray*} (-)_{QMB}:
A_{2-\mathrm{braided}}&\rightarrow& A_{\mathcal{QMB}},
\end{eqnarray*}by $(V)_{QMB}=(V, M,\sigma)$, where $M$ is the quantum multi-brace algebra constructed from $(V,\ast,\cdot,\sigma)$, for any $(V,\ast,\cdot,\sigma)\in A_{2-\mathrm{braided}}$.

By the above proposition, we have immediately that:
\begin{corollary}Let $(V,M,\sigma)$ be a quantum multi-brace algebra and $(T(V),\ast,m,\beta)$ be the 2-braided algebra with product $\ast=\varepsilon\otimes \varepsilon+\sum_{n\geq 1}M^{\otimes n}\bigtriangleup_\beta^{(n-1)}$ and $m$ the concatenation. Then the inclusion $i:V\rightarrow T(V)$ is a quantum multi-brace algebra morphism, i.e., $i\circ M_{pq}=\overline{M}_{pq}\circ(i^{\otimes p}\otimes i^{\otimes q})$, for any $p,q\geq 0$. Here $\overline{M}_{pq}$ is the quantum multi-brace algebra structure on $T(V)$ defined above.\end{corollary}

\section*{Acknowledgements}
The first author is grateful to Prof. Yi Zhang for his many helps
during last six years. Both of the authors would like to thank the
referee for careful reading and useful comments. In particular,
the terminology "multi-brace algebra" is suggested by the referee.
The first author was partially supported by a CNRS fellowship from
L'Ambassade de France en Chine, NSF of China (Grant No. U0935003),
and the grant for new teachers from DGUT (Grant No. ZJ100501).
\bibliographystyle{amsplain}

\end{document}